\documentclass[12pt]{amsart}
\usepackage{amssymb}
\usepackage{}
\usepackage[top=1.2in, left=1.2in, right=1.2in, bottom=1.2in]{geometry}
\usepackage[utf8x]{inputenc}
\usepackage{amsfonts}
\usepackage{amssymb,epsfig}
\usepackage{amsmath}
\usepackage{delarray}
\usepackage{graphicx}
\usepackage{xcolor}
\usepackage{tensor}
\usepackage[linktocpage=true]{hyperref}

\usepackage{bookmark}

\newtheorem{thm}{Theorem}[section]
\newtheorem{theorem}[thm]{Theorem}
\newtheorem*{theorem*}{Theorem}

\newtheorem{corollary}[thm]{Corollary}

\newtheorem{lemma}[thm]{Lemma}

\newtheorem{proposition}[thm]{Proposition}
\theoremstyle{definition}
\newtheorem{definition}[thm]{Definition}

\theoremstyle{remark}

\numberwithin{equation}{section}
\newtheorem{example}[thm]{Example}

\newcommand{\RR}{\mathbb{R}}
\newcommand{\CC}{\mathbb{C}}
\newcommand{\NN}{\mathbb{N}}
\newcommand{\QQ}{\mathbb{Q}}
\newcommand{\KK}{\mathbb{K}}

\newcommand{\ZZ}{\mathbb{Z}}

\newcommand{\Arg}{\mathop{\rm Arg}}

\newcommand{\intprod}{\!\mathbin{\hbox to 6pt
{\vrule height0.4pt width5pt depth0pt \kern-.4pt
\vrule height6pt width0.4pt depth0pt\hss}}\!}

\newcommand{\re}[1]{(\ref{#1})}

\newcommand{\rt}[1]{Theorem~\ref{#1}}

\def\red#1{{\color{red} #1  }}

\def\beq{\begin{equation}}
\def\eeq{\end{equation}}
\begin{document}

\title{Complete integrability of diffeomorphisms and their local normal forms}

\author{Kai Jiang}
\address{Beijing international center for mathematical research, Peking University\\and Laboratoire J.A. Dieudonn\'e,
UMR CNRS 7351 Universit\'e de Nice Sophia-Antipolis}
\email{kai.jiang@bicmr.pku.edu.cn}

\author{Laurent Stolovitch}
\address{CNRS and Laboratoire J.A. Dieudonn\'e,UMR CNRS 7351 Universit\'e de Nice Sophia-Antipolis}
\email{stolo@unice.fr}

\thanks{This work has been supported by the French government, through the UCAJEDI Investments in the Future project managed by the National Research Agency (ANR) with the reference number ANR-15-IDEX-01 and by ANR project BEKAM with the reference number ANR-15-CE40-0001-03}

\begin{abstract}
	In this paper, we consider the normal form problem of a commutative family of germs of diffeomorphisms at a fixed point, say the origin,  of $\KK^n (\KK=\CC \text{~or~} \RR)$. We define a notion of {\it integrability} of such a family.
	We give sufficient conditions which ensure that such an integrable family can be transformed into a normal form by an analytic (resp. a smooth)  transformation 
	if the initial diffeomorphisms are  analytic (resp. smooth).
\end{abstract}

\date{Version5c, Dec. 2019}

\maketitle
\hspace{9cm}{\it A la mémoire de Walter}\\

\section{Introduction}
When studying dynamical systems with continuous time (i.e. systems of differential equations) or discrete time (i.e. diffeomorphisms), special  solutions, such as fixed points also called equilibrium points, attract a lot of attention. In particular, one needs to understand the behavior of nearby solutions. This usually requires some deep analysis involving {\it normal forms} \cite{Arn-geom}, which are models supposed to capture the very nature of the dynamics to which the initial dynamical system is conjugate.
When considering analytic or smooth dynamical systems, one needs extra assumptions in order to really obtain dynamical and geometrical information on the initial dynamical system via its normal form. These assumptions can sometimes be understood as having a lot of symmetries. This led to the concept of {\it integrability}. 

In the framework of differential equations or vector fields\red{,} a first attempt to define such a notion for Hamiltonian systems is due to Liouville  \cite{Liouville1855}. This led, much later, to the now classic Liouville-Mineur-Arnold theorem \cite{Arn1} which provides action-angle coordinates by a canonical transformation. For a general concept of action-angle coordinates we refer to \cite{Zung2018}.
In 1978, J. Vey studied in the groundbreaking work \cite{Vey1978}, a family of $n$ Poisson commuting analytic Hamiltonian functions 
in a neighborhood of a common critical point. Under a generic condition on their Hessians,  he proved that the family can be simultaneously transformed into a (Birkhoff) normal form. This system of Hamiltonian has to be understood as ``completely integrable system''. Later, H. Elliason, H. Ito, L. Stolovitch, N.T.Zung 
to name a few, generalized or improved J. Vey's theorem in different aspects including non-Hamiltonian setting \cite{Eliasson1984, Eliasson1990, Ito1989, Ito1992, Stolovitch2000, Stolovitch2005, Zung2005}. This has been recently developed in the context of PDE's as infinite dimensional dynamical systems \cite{kappeler-poschel-book,kuksin-perelman, BS20}. In a different context of global dynamics, a notion of "integrable maps" has been devised relative to long-time behaviour of their orbits and their {\it complexity} \cite{veselov-intg-Russian,veselov-CMP}.

In \cite{Bogoyavlenskij1998}, a new integrability condition for non-Hamiltonian vector fields was established, which involves commuting vector fields and common first integrals. Concretely, such an integrable system on an $n$-dimensional manifold consists of $p$ independent commuting vector fields and $n-p$ functionally independent common first integrals. For a local integrable system near a common equilibrium point of the vector fields, the $p$ vector fields (resp. $n-p$ first integrals) may be not always independent (resp. functionally independent), so they are required to be independent (resp. functionally independent) almost everywhere.
Then one can seek for a simultaneous Poincar\'e-Dulac normal form (named ``normal form'' for short) of the vector fields. Such a transformation can be obtained under certain non-degeneracy conditions \cite{Stolovitch2000,Zung2015,Jiang2016}.

We aim at considering, in the same spirit, discrete dynamical systems given by a family of germs of commuting diffeomorphisms at a fixed.  On the one hand, the simultaneous linearization of such  holomorphic family under an appropriate "small divisors" condition has been treated by the second author \cite{stolo-bsmf}. On the other hand and to the best of our knowledge, the only known result in this spririt, related to "integrability of diffeomeorphisms" is due to X. Zhang \cite{Zhang2013} who considered a single diffeomorphism near a fixed point. On the other hand,  In this article, we propose an analogue notion of integrability of a family of commuting diffeomorphisms near a common fixed point, then we  explore their local behavior and study their normal forms. 

In this paper, we consider local diffeomorphisms on $(\KK^n,0)$  having the form
\begin{equation}
\label{eq:11}
\Phi(x)=Ax+\textit{higher order terms}
\end{equation}
such that the coefficient matrix $A$ of the linear part at the origin has a (real or complex) logarithm, i.e., there exists a matrix $B$ such that $A=e^B$. It is known that a complex matrix has a logarithm if and only if it is invertible \cite{Gantmakher1959};  a real matrix has a real logarithm if and only if it is invertible and each Jordan block belonging to a negative eigenvalue occurs an even number of times \cite{Culver1966}.

Let $\Phi$ be a germ of diffeomorphism near a fixed point, say the origin. Then, for any integer $k\geqslant 1$, $\Phi^{(k)}$ denotes the homogeneous polynomial of degree $k$ of the Taylor expansion at the origin of $\Phi$.

\begin{definition}[Integrability, local version]
  Let $\Phi$ be a local diffeomorphism on $\KK^n (\KK=\CC \text{~or~} \RR)$  having the origin $0$ as its isolated fixed point. If there exists $p\geqslant1$ pairwise commuting (germs of) diffeomorphisms $\Phi_1=\Phi,\Phi_2,\ldots,\Phi_p$ of the form \eqref{eq:11} with $D\Phi_i(0)=A_i$ and $q=n-p$ common first integrals $F_1,\ldots,F_q$ of the diffeomorphisms such that
   \begin{itemize}
	 \item the diffeomorphisms are independent in the following sense:  the matrices $\{\ln A_i\}_{i=1,\ldots, p}$ are linearly independent over $\mathbb{K}$; if $\KK=\CC$ then $\ln A_i$ are not unique and we require the independence of  families of all possible logarithms;
     \item the first integrals are functionally independent almost everywhere, i.e., the wedge of differentials of the first integrals satisfies $dF_1\wedge\cdots\wedge dF_q\neq0$ almost everywhere,
   \end{itemize}
  then  $\Phi$ is called a \textbf{completely integrable} diffeomorphism and we say \break $(\Phi_1=\Phi,\Phi_2,\ldots,\Phi_p,F_1,\ldots,F_q)$ is a (discrete) completely integrable system of type $(p,q)$.
\end{definition}

We remark that analytic integrable systems in $n$-dimensional Euclidean spaces containing a single diffeomorphism and $n-1$ functionally independent first integrals were studied \cite{Zhang2008,Zhang2013} and a local normal form was obtained under a mild generic condition.

We now introduce the notion of non-degeneracy of integrable diffeomorphisms.
\begin{definition}[non-degeneracy]
  We say that a local discrete integrable system \break $(\Phi_1,\ldots,\Phi_p,F_1,\ldots,F_q)$ of type $(p,q)$ is \textbf{non-degenerate}, if for all $i=1,\ldots,p$,
  \begin{itemize}
    \item the linear part $\Phi_i^{(1)}$ of the diffeomorphism $\Phi_i$ at the origin is semi-simple, i.e., write $\Phi_i^{(1)}(x)=A_i x$, then the coefficient matrix $A_i$ is diagonalizable over $\CC$;
    \item there exist $q$ functionally independent homogeneous polynomials $P_1,\ldots,P_q$ such that $(\Phi_1^{(1)},\ldots,\Phi_p^{(1)},P_1,\ldots,P_q)$ is a linear (discrete) completely integrable system of type $(p,q)$;
  \end{itemize}
\end{definition}

The notion of the non-degeneracy of commuting diffeomorphisms follows that of non-degeneracy of commuting vector fields defined in \cite{Zung2015}. The first condition is generic in the sense that almost all matrices are diagonalizable over $\CC$; and the second condition is automatically satisfied for formal or analytic integrable systems by Ziglin's lemma \cite{Ziglin1982}.

If there exist logarithms $\ln A_i$ of $A_i$ such that any common first integral of $\Phi_1^{(1)},\ldots,\Phi_p^{(1)}$ is also a common first integral of the linear vector fields $X_i$ defined by $\ln A_i$, then $X_1,\ldots,X_p$ together with $P_1,\ldots,P_q$ form a linear non-degenerate integrable system of type $(p,q)$. In such a case, the family of (linear) integrable diffeomorphisms is said to be {\it infinitesimally integrable} and $X_1,\ldots,X_p$ are called {\it infinitesimal (linear) generators} in the sense $\Phi_i^{(1)}=e^{X_i}$, and we pick and fix one such family of vector fields if the logarithms are not unique.

The following example shows that not all linear integrable diffeomorphisms is infinitesimally integrable.
\begin{example}
\label{ex:13}
The integrable system  $\Phi(x,y)=(-2x,\frac{1}{2}y),F=x^2y^2$ on $\CC^2$ of type $(1,1)$. The corresponding vector field $X=(\ln2+(2K_1+1)\sqrt{-1}\pi)x\frac{\partial}{\partial x}-(\ln2+2 K_2\sqrt{-1}\pi) y\frac{\partial}{\partial y}$ does not admit any homogeneous first integral for all integers $K_1, K_2$. Indeed, if $X(x^py^q)=0$ for some natural integers $p,q$, then we would have
$(\ln 2)(p-q)+2\sqrt{-1}\pi[(K_1+\frac{1}{2})p-K_2q]=0$. Then vanishing of the real part leads to $p=q$ so that the vanishing of the imaginary part reads $(K_1-K_2+\frac{1}{2})p=0$; this is not possible.
\end{example}

\section{Preliminaries and formal normal forms}


In this section, we introduce some notions and lemmas in order to well organize the proof of the main theorem.
The first lemma is analogue to the Poincar\'e-Dulac normal form for commuting vector fields. It requests neither integrability nor non-degeneracy.
\begin{lemma}[Th\'eor\`eme 4.3.2 in \cite{Chaperon1986Asterisque}]
\label{lem:PD-NF}
  Let $\Phi_1,\Phi_2,\ldots,\Phi_p$ be $p$ commuting diffeomorphisms in $\KK^n$ around $0$.  Let $\Phi_{j}^{ss}$ be the semi-simple part of the Jordan decomposition of the linear part of $\Phi_i$ at the origin. There exists a formal transformation  $\hat\Psi$ such that,  $\hat \Phi_{i}\circ\Phi_{j}^{ss}=\Phi_{j}^{ss}\circ\hat\Phi_{i}$ for all $i,j=1,2,\ldots,p$, where  $\hat \Phi_{i}:=\hat\Psi^{-1}\circ\Phi_i\circ\hat\Psi$. We say the diffeomorphisms are in Poincar\'e-Dulac normal form. Moreover, when $\KK=\CC$, let $\rho$ be an anti-holomorphic involution. Assume that $\Phi_i\rho=\rho\Phi_i$ for all $i$, then $\hat\Psi$ can be well chosen such that $\hat\Psi\rho=\rho\hat\Psi$ as well, and we call it $\rho$-equivariant normalization. 
\end{lemma}
Though the result can be obtained by direct computation, it is easier to understand \cite{Chaperon1986Asterisque} via the Jordan decomposition theorem. For completeness, we provide a proof (in particular, of the $\rho$-equivariant case used in section 5) here. 

\begin{proof}[Idea of a proof]
For each positive integer $\ell$, let $\mathcal{E}^{(\ell)}$ denote the $\KK$ algebra of $\ell$-th order jets (Taylor expansions) $j_0^\ell f$ at $0$ of smooth $\KK$ functions $f$ on $\KK^n$ and let $\mathcal D^{(\ell)}$ be the group of $\ell$-th order jets $j_0^\ell\Phi$ at $0$ of smooth diffeomorphisms $\Phi$  vanishing at $0\in\KK^n$; then, the map which sends $j_0^\ell\Phi
\in\mathcal D^{(\ell)}$  
to 
$(j_0^\ell\Phi)^*:j_0^\ell f\mapsto j_0^\ell(f\circ\Phi)$ is an isomorphism of $\mathcal D^{(\ell)}$ onto the group $\mathrm{aut}(\mathcal E^{(\ell)})$ of automorphisms of $\mathcal E^{(\ell)}$. We may sometimes abuse notations of jets as its representative for simplicity.

Thus, 
$(j_0^\ell\Phi_1)^*,\ldots,(j_0^\ell\Phi_p)^*$ are commuting elements of $\mathrm{aut}(\mathcal E^{(\ell)})$; it follows that their semi-simple parts 
, as endomorphisms of the $\KK$ vector space $\mathcal E^{(\ell)}$, commute  pairwise. Now, it is easy to see that the semi-simple part of an element of $\mathrm{aut}(\mathcal E^{(\ell)})$ lies in $\mathrm{aut}(\mathcal E^{(\ell)})$ by the Jordan-Chevalley theorem; therefore, there exist pairwise commuting elements $j_0^\ell S_i$ of $\mathcal D^{(\ell)}$ such that 
 $(j_0^\ell S_i)^*$ is the semi-simple part of $(j_0^\ell\Phi_i)^*$ for $1\leq i\leq p$.

Then, the following two facts are not difficult to establish:
   \begin{itemize}
  \item one has $j_0^1S_i=\Phi_i^{ss}$ for $1\leq i\leq p$
  \item as the $(j_0^\ell S_i)^*$'s are commuting elements of $\mathrm{aut}(\mathcal E^{(\ell)})$, their semi-simplicity implies, essentially by definition, that the $j_0^\ell S_i$'s can be simultaneously linearized by a formal diffeomorphism $j_0^\ell\Psi$ of order $\ell$. 
\end{itemize}

Indeed, the diffeomorphism $j_0^\ell\Psi$ can be defined through the transformation $(j_0^\ell\Psi)^*$ that normalizes the commuting family $\{(j_0^\ell S_1)^*,\ldots,(j_0^\ell S_p)^*\}$: (after complexification if necessary,) let us take $n$ common eigenvectors $j_0^\ell f_1,\ldots,j_0^\ell f_n$ of $(j_0^\ell S_1)^*,\ldots,(j_0^\ell S_p)^*$ such that $j_0^1 f_1,\ldots,j_0^1 f_n$ form a basis of $\mathcal{E}^{(1)}$ and let us set $(j_0^\ell S_i)^* (j_0^\ell f_m)=\lambda_{im}j_0^\ell f_m $. Let us define $(j_0^\ell\Psi)^*$  by sending $(j_0^\ell\Psi)^*(j_0^\ell(f_m^{(1)}))=j_0^\ell f_m$ for $m=1,\ldots,n$ where $f_m^{(1)}$ denotes the linear part of $f_m$. It follows from the equations $(j_0^\ell S_i)^* (j_0^\ell\Psi)^*\left(j_0^\ell(f_m^{(1)})\right)=\lambda_{im}(j_0^\ell\Psi)^*\left(j_0^\ell(f_m^{(1)})\right)$ that $(j_0^\ell\Psi)^*$ normalizes (that is, diagonalizes or block-diagonalizes) $(j_0^\ell S_i)^*$'s and therefore $j_0^\ell\Psi$ linearizes $j_0^\ell S_i$'s.

This change of coordinates simultaneously transforms the diffeomorphisms $\Phi_i$'s into a Poincar\'e-Dulac normal form to order $\ell$, that is, for all $i,j\in\{1,\ldots,p\}$,
\[
\Phi_i^{ss}\circ j_0^\ell\Phi_j=j_0^\ell\Phi_j\circ\Phi_i^{ss}.
\]
Take the inverse limit $\ell$ tends to $\infty$ and we get a formal transformation $\hat\Psi:=j_0^\infty\Psi=\displaystyle\lim_{\substack{\longleftarrow\\ \ell\rightarrow\infty}} j_0^\ell\Psi$ after which the above equations hold for all natural number $\ell$, i.e., $\Phi_{i}\circ\Phi_{j}^{ss}=\Phi_{j}^{ss}\circ\Phi_{i}$ in the formal sense.

Now assume that $\KK=\CC$ and that the anti-holomorphic involution $\rho$ commutes with all $\Phi_i$'s. As $\rho \tilde{\Phi}_i\rho= \tilde{\Phi}_i$, we have $\rho (j_0^\ell\tilde{\Phi}_i)\rho:=j_0^\ell(\rho \tilde{\Phi}_i\rho)= j_0^\ell\tilde{\Phi}_i$ and then $\left(\rho (j_0^\ell\tilde{\Phi}_i)\rho\right)^*=(j_0^\ell\tilde{\Phi}_i)^*$. It follows from the uniqueness of Jordan-Chevalley decomposition that $(\rho \tilde{S}_i\rho)^*=(\tilde{S}_i)^*$ where $\tilde{S}_i\in\mathcal{D}^{(\ell)}$ and $\tilde{S}_i^*$ is the semi-simple part of $(j_0^{\ell} \tilde{\Phi}_i)^*$ in $Aut(\mathcal{E}^{(\ell)})$. Therefore, for any common eigenvector $j_0^\ell f$ of the $\tilde{S}_i^*$'s, we set $\tilde{S}_i^* j_0^\ell f=\lambda_i j_0^\ell f$. Let $c$ denote the conjugate of complex vectors. Hence, 
$j_0^\ell (c f \rho)=:c (j_0^\ell f) \rho$ is also a common eigenvector of the $\tilde{S}_i^*$'s with respect to the eigenvalue $\bar\lambda_i$. Indeed,  on the one hand, we have
\[
 \tilde{S}_i^* \left(j_0^\ell (c f\rho)\right)
 =(\rho \tilde{S}_i\rho)^* j_0^\ell(c f\rho )
 =j_0^\ell(c f\rho\rho\tilde{S}_i\rho)
 =j_0^\ell(c f\tilde{S}_i\rho),
\]
on the other hand, we have
\[
 \bar\lambda_i j_0^\ell(c f\rho)
 =j_0^\ell(c \lambda_i f\rho)
 =c \left(\lambda_i (j_0^\ell f)\right)\rho
 =c (\tilde{S}_i^*j_0^\ell f)\rho
 =c j_0^\ell( f\tilde{S}_i)\rho
 =j_0^\ell(c f\tilde{S}_i\rho).
\]
Recall that  $(j_0^\ell \Psi)^*$  is defined with the help of eigenvectors $j_0^\ell f_1,\ldots,j_0^\ell f_n$ such that $j_0^1 f_1,\ldots,j_0^1 f_n$ are independent. 
Then one can verify $j_0^\ell(\rho\Psi\rho)=j_0^\ell \Psi$ directly since $(j_0^\ell(\rho\Psi\rho))^*$ also sends $j_0^\ell f_m^{(1)}$ to $j_0^\ell f_m$ for $m=1,\ldots,n$ as $(j_0^\ell \Psi)^*$ does: as
\[
 j_0^\ell(c f_m^{(1)}\rho\Psi)
 =(j_0^\ell \Psi)^* j_0^\ell(c f_m^{(1)}\rho)
 =j_0^\ell(cf_m\rho)
 =c j_0^\ell f_m \rho,
\]
we have
\[
 (j_0^\ell(\rho\Psi\rho))^* j_0^\ell f_m^{(1)}
 =j_0^\ell(f_m^{(1)}\rho\Psi\rho)
 =j_0^\ell(ccf_m^{(1)}\rho\Psi\rho)
 =c j_0^\ell(c f_m^{(1)}\rho\Psi)\rho
 =c (c j_0^\ell f_m \rho) \rho
 =j_0^\ell f_m.
\]
Hence we have $\rho\Psi\rho=\Psi$ by the inverse limit.

\end{proof}

Assuming that the  semi-simple linear part $\Phi_i^{ss}$ of $\Phi_i$ is diagonal, we set
\[
  \Phi_i^{ss}(x_1,\ldots,x_n)=(\mu_{i1}x_1,\ldots,\mu_{in}x_n).
\]
Let us write the homogeneous part of order $\ell$ of $\Phi_j$ as $\Phi_{j}^{(\ell)}=(\phi_{j1}^{(\ell)},\ldots,\phi_{jn}^{(\ell)})$, then we can express  $\Phi_{j}^{(\ell)}\circ\Phi_{i}^{ss}=\Phi_{i}^{ss}\circ\Phi_{j}^{(\ell)}$ in local coordinates, that is, for any $m\in\{1,\ldots,n\}$, we have
\[
\phi_{jm}^{\ell}(\mu_{i1}x_1,\ldots,\mu_{in}x_n)=\mu_{im}\phi_{jm}^{\ell}(x_1,\ldots,x_n).
\]
It follows that for any $j$, the indices $(\gamma_1,\ldots,\gamma_n)$ of every monomial term $x_1^{\gamma_1}\cdots x_n^{\gamma_n}$ in the $m$-th component $\phi_{jm}$ of $\Phi_j$ satisfies the following \textbf{resonant equations} with respect to the $m$-th component
\begin{equation}
\label{eq:21}
 \mu_{im}=\prod_{k=1}^n \mu_{ik}^{\gamma_k},\quad i=1,\ldots,p.
\end{equation}
We denote by  $\mathcal R_m$ the set of solutions $(\gamma_1,\ldots,\gamma_n)$ with $\gamma_k$ natural numbers and $\sum_{k=1}^n\gamma_k\geqslant2$.

Conversely, it is easy to see that the commuting diffeomorphisms $\Phi_i=(\phi_{i1},\ldots,\phi_{in})$ are formally in the Poincar\'e-Dulac normal form if the Taylor expansion of $\phi_{im}$ contains only resonant terms with respect to the $m$-th component, i.e., the indices of every monomial term lie in $\mathcal R_m$.

We now turn to first integrals of a diffeomorphism already in the Poincar\'e-Dulac normal form. The second lemma is also a parallel version from the result for vector fields: A first integral of a vector field in the Poincar\'e-Dulac normal form is also a formal first integral of the semi-simple linear part of the vector field \cite{Walcher1991}.

Let us recall first integral relations for linear diffeomorphisms before stating our lemma. Given a semi-simple linear diffeomorphism $\Phi(x_1,\ldots,x_n)=(\mu_{1}x_1,\ldots,\mu_{n}x_n)$, the equation
\begin{equation}
\label{eq:22}
  \mu_1^{\ell_1}\cdots\mu_n^{\ell_n}=1
\end{equation}
with respect to the non-negative integers $\ell_1,\ldots,\ell_n$ is called the \textbf{first integral equation} for $\Phi$. We denoted by $\Omega$ the set of the solutions of the first integral equation~:
\begin{equation}\label{Omega}
\Omega:=\left\{(\ell_1,\ldots,\ell_n)\in\mathbb{N}^n:\; \mu_1^{\ell_1}\cdots\mu_n^{\ell_n}=1\right\}.
\end{equation}
Hence, $\{x_1^{\ell_1}\cdots x_n^{\ell_n}:(\ell_1,\ldots,\ell_n)\in\Omega\}$ is the set of all monomial first integrals of $\Phi$ up to multiplication by constant coefficients.
For $p$ commuting linear diffeomorphisms, we will consider $p$ first integral equations simultaneously and the set of their common solutions is still denoted by $\Omega$.

\begin{lemma}\label{Lem:NFofFI}
  Assume $\Phi_1,\Phi_2,\ldots,\Phi_p$ are in the Poincar\'e-Dulac normal form, then formal first integrals of the diffeomorphisms are formal first integrals of the semi-simple parts of the diffeomorphisms.
\end{lemma}
\begin{proof}
  Write the semi-simple part of $\Phi_i$ as
  \[
  \Phi_i^{ss}(x_1,\ldots,x_n)=(\mu_{i1}x_1,\ldots,\mu_{in}x_n),
  \]
  then the lemma to prove is that if we consider the Taylor expansion of a first integral of the diffeomorphisms, then the indices of every monomial term lie in $\Omega$ provided that the diffeomorphisms are in the Poincar\'e-Dulac normal form.



  Assume $F$ is a common first integral of the diffeomorphisms and  let $F^{(low)}$ be the homogeneous part of lowest degree of $F$. Consider the homogeneous part of lowest degree of both sides of the equations $F\circ\Phi_j=F$, we have obviously $F^{(low)}\circ\Phi_j^{(1)}=F^{(low)}$. View $\Phi_j^{(1)}$ as a linear operator on the space of homogeneous polynomials of degree \textit{low} mapping $f$ to $f\circ\Phi_j^{(1)}$, then $F^{(low)}$  is in the eigenspace belonging to eigenvalue $1$, and therefore it is in the eigenspace belonging to the eigenvalue $1$ of the semi-simple part $\Phi_j^{ss}$, i.e., $F^{(low)}\circ\Phi_j^{ss}=F^{(low)}$.

  We claim that the homogeneous part of any degree of $F$ is also a common first integral of $\Phi_j^{ss}$. Now assume the claim is true for homogeneous parts of degree less than $\ell$, which means that any monomial term $cx_1^{\ell_{1}}\cdots x_n^{\ell_{n}}$ in $F$ with $\ell_1+\cdots+\ell_n<\ell$ has indices $(\ell_1,\ldots,\ell_n)$ in $\Omega$. Consider the homogeneous part of degree $\ell$ of both sides of $F\circ\Phi_j=F$, we have
  \begin{equation}
  \label{eq:23}
  F^{(\ell)}\circ\Phi_j^{(1)}+(F^{(<\ell)}\circ\Phi_j)^{(\ell)}=F^{(\ell)},
  \end{equation}
  where $F^{(<\ell)}$ denotes the part of $F$ with degree less than $\ell$.
	By our inductive hypothesis, $(F^{(<\ell)}\circ\Phi_j)^{(\ell)}$ is a common first integral of $\Phi_i^{ss}$ since
	$$
	(F^{(<\ell)}\circ\Phi_j)^{(\ell)}\circ \Phi_i^{ss}=(F^{(<\ell)}\circ\Phi_j\circ \Phi_i^{ss})^{(\ell)}= (F^{(<\ell)}\circ\Phi_i^{ss}\circ\Phi_j)^{(\ell)}=(F^{(<\ell)}\circ\Phi_j)^{(\ell)}.
	$$
	Therefore, by \eqref{eq:23} we have
  \[
  (F^{(\ell)}-F^{(\ell)}\circ\Phi_j^{(1)})\circ\Phi_i^{ss}=F^{(\ell)}-F^{(\ell)}\circ\Phi_j^{(1)},
  \]
  equivalently,
  \[
    (F^{\ell}-F^{(\ell)}\circ\Phi_i^{ss})\circ\Phi_j^{(1)}=F^{\ell}-F^{(\ell)}\circ\Phi_i^{ss}.
  \]
  Considering composition by $\Phi_j^{(1)}$ on the right as a linear operator on the space of homogeneous polynomials of degree $\ell$, we have that $F^{(\ell)}-F^{(\ell)}\circ\Phi_i^{ss}$ lies in the eigenspace belonging to the eigenvalue $1$ of $\Phi_j^{(1)}$ and then of $\Phi_j^{ss}$. For any $i,j\in\{1,\ldots,p\}$, we have
  \begin{equation}
  \label{eq:24}
  (F^{(\ell)}-F^{(\ell)}\circ\Phi_i^{ss})\circ\Phi_j^{ss}=F^{(\ell)}-F^{(\ell)}\circ\Phi_i^{ss},
  \end{equation}
 and it  implies that $F^{(\ell)}$ is also a first integral of $\Phi_i^{ss}$ for $i=1,\ldots,p.$
 Indeed, assume on the contrary that a monomial term $cx_1^{\ell_1}\cdots x_n^{\ell_n}$ in $F^{(\ell)}$ is not a first integral of $\Phi_i^{ss}$, then $F^{(\ell)}-F^{(\ell)}\circ\Phi_i^{ss}$ contains a non-vanishing term $c(1-\prod_{k=1}^n \mu_{ik}^{\ell_k})\prod_{k=1}^nx_k^{\ell_k}$, and then $(F^{(\ell)}-F^{(\ell)}\circ\Phi_i^{ss})-(F^{(\ell)}-F^{(\ell)}\circ\Phi_i^{ss})\circ\Phi_i^{ss}$ contains a non-vanishing term
$$
 c(1-\prod_{k=1}^n \mu_{ik}^{\ell_k})^2\prod_{k=1}^nx_k^{\ell_k},
$$
which contradicts with equation \eqref{eq:24}.

\end{proof}

%
%

\begin{definition}
Given a family of $p$ commuting linear vector fields $X_1,\ldots,X_p$ on $(\KK^n,0)$ and assume their semi-simple liner parts read $X_i^{ss}=\sum_{m=1}^n\lambda_{im} x_m\frac{\partial}{\partial x_m}$, we say it is \textbf{weakly resonant} with respect to first integrals if there exists integers $k_1,\ldots,k_n$  such that
\[
 (\sum_{m=1}^n k_m\lambda_{1m},\ldots,\sum_{m=1}^n k_m\lambda_{pm})\in 2\sqrt{-1}\pi\ZZ^p-\{0\}.
\]

We say the family of  commuting linear vector fields to be \textbf{weakly non-resonant} if there do not exist such integers $k_1,\ldots,k_n$.

Given a family of $p$ commuting  diffeomorphisms on $(\KK^n,0)$, we say it is weakly resonant (resp. weakly non-resonant) if the family of infinitesimal generators  of their semi-simple liner parts is (resp. is not).
\end{definition}

We emphasize that the family of $X_1,\ldots,X_p$ can be weakly non-resonant and resonant as well since we could have $(\sum_{m=1}^n k_m\lambda_{1m},\ldots,\sum_{m=1}^n k_m\lambda_{pm})=0$. 
We also remark that our notion of weak resonance with respect to first integrals is slightly different from that with respect to vector fields: the latter requires the existence of integers $k_1,\ldots,k_n$ such that they are no less than $-1$ and there is at most one integer equal to $-1$, see \cite{Li-Llibre-Zhang2002} for example. 

\begin{definition}
Let $\Phi_1,\ldots,\Phi_p$ be $p$ commuting linear diffeomorphisms on $(\KK^n,0)$ and assume the eigenvalues of the semi-simple linear part $\Phi_i^{ss}$ of $\Phi_i$ are $\mu_{i1},\ldots,\mu_{in}$. We say the family of diffeomorphisms is \textbf{projectively hyperbolic} if the $p$ real vectors $(\ln|\mu_{i1}|,\ldots,\ln|\mu_{in}|)$ are $\RR$-linearly independent.
\end{definition}

We recall that the family $\{\Phi_i\}$ is said to be {\it hyperbolic} if any $p$ of the $n$ covectors $(\ln|\mu_{1j}|,\cdots, \ln|\mu_{pj}|)$ are linearly independent (which coincides with the usual meaning if $p = 1$).
By definition, the projection of a projectively hyperbolic family of $p$ linear diffeomorphisms onto some $p$-dimensional subspace is hyperbolic~: the $p$ by $n$ real matrix $(\ln|\mu_{im}|)$ has full rank and therefore there exist $p$ columns, say the $(m_1,\ldots,m_p)$-th columns, which are linearly independent, then the projection of the diffeomorphisms onto the subspace of $(x_{m_1},\ldots,x_{m_p})$ form a hyperbolic family of diffeomorphisms in the sense of definition \ref{def:44}.  Particularly, for a single (linear) diffeomorphism, it is projectively hyperbolic if and only if there exists at least one eigenvalue that does not lie on the unit circle.
\begin{example}
\label{ex:25}
  The diffeomorphism $(\Phi(x,y)=(e^{\sqrt{-1}}x,e^{-\sqrt{-1}}y)$  is not projectively hyperbolic.
\end{example}

With the notions above, we can now state our theorem.

\begin{theorem}
\label{thm:FormalNormalForm}
  Let $(\Phi_1=\Phi,\Phi_2,\ldots,\Phi_p,F_1,\ldots,F_q)$ be a formal non-degenerate discrete integrable system of type $(p,q)$ on $\KK^n$ at a common fixed point, say the origin $0$. Assume that the linear part of $\Phi_i$, at the origin reads $\Phi_i^{(1)}(x_1,\ldots,x_n)=(\mu_{i1}x_1,\ldots,\mu_{in}x_n)$, for all $i=1,\ldots,p$. If the family \{$\Phi_i^{(1)}\}$ is either projectively hyperbolic or infinitesimally integrable with a weakly non-resonant family of generators, then there is formal diffeomorphism, tangent to Identity, which conjugates each diffeomorphisms $\Phi_i$, $i=1,\ldots,p$ to
\begin{equation}
\label{good-nf}
    \hat\Phi_i=(\mu_{i1}x_1(1+\hat\varphi_{i1}),\ldots,\mu_{in}x_n(1+\hat\varphi_{in})).
\end{equation}
  Here, the $\hat\varphi_{ik}$'s are not only common first integrals of $\Phi_i^{ss}$ (this turns $\hat\Phi_i$ into a Poincar\'e-Dulac normal form) but also they satisfy 
	\beq\label{intgnf}
	\prod_{k=1}^n (1+\hat\varphi_{ik})^{\gamma_k}=1
	\eeq 
	for all $(\gamma_1,\ldots,\gamma_n)$ in the set $\Omega$ (defined by \re{Omega}).
\end{theorem}

We give a remark that the diffeomorphism in example \ref{ex:13} is projectively hyperbolic but has no infinitesimally integrable generator; example \ref{ex:25} provides an example of a diffeomorphism which is not projectively hyperbolic but which is infinitesimally integrable with  a weakly non-resonant generator $X=\sqrt{-1}x\frac{\partial}{\partial x}-\sqrt{-1}y\frac{\partial}{\partial y}$;
and the system in example \ref{ex:34} in the next section satisfies none of the two conditions.


\section{Proof of the theorem}
\begin{lemma}
\label{lem:VectOmega}
  Let $(\Phi_1,\Phi_2,\ldots,\Phi_p,F_1,\ldots,F_q)$ be a non-degenerate integrable system in which the diffeomorphisms are in the Poincar\'e-Dulac normal form and their linear parts read $\Phi_i^{(1)}(x_1,\ldots,x_n)=(\mu_{i1}x_1,\ldots,\mu_{in}x_n)$ for $i=1,\ldots,p$. Let $\mathit{Vect}_\KK\Omega$ be the vector space spanned by $\Omega$ over $\KK$. We have the dimension of $\mathit{Vect}_\KK\Omega$ is no less than $q$;  if the family \{$\Phi_i^{(1)}\}$ is either projectively hyperbolic or infinitesimally integrable with a weakly non-resonant family of generators, then the dimension is equal to $q$.
\end{lemma}
\begin{proof}
  By the non-degeneracy condition, we have homogeneous polynomials $P_1,\ldots,P_q$ which are common first integrals of $\Phi_i^{ss}$ and the corresponding vector fields. Then every monomial term is a common first integral with indices in $\Omega$. 
  As $P_1,\ldots,P_q$ are functionally independent almost everywhere, i.e., $dP_1\wedge\cdots\wedge dP_q\neq0$, there exist monomial terms $G_j=x_1^{\ell_{j1}}\cdots x_n^{\ell_{jn}}$ (ignore coefficient) of $P_j$ such that $dG_1\wedge\cdots\wedge dG_q\neq0$, equivalently,
  \[
   \sum_{1\leqslant k_1<\ldots<k_q\leqslant n}\det
    \begin{pmatrix}
     \frac{\partial G_1}{\partial x_{k_1}}&\cdots&\frac{\partial G_1}{\partial x_{k_q}}\\
     \vdots&&\vdots\\
     \frac{\partial G_q}{\partial x_{k_1}}&\cdots&\frac{\partial G_q}{\partial x_{k_q}}
    \end{pmatrix}
   dx_{k_1}\wedge\cdots dx_{k_q}\neq0.
  \]
  It implies at least one determinant (as coefficient) in the above inequality is nonzero, that is, there exist $k_1<\ldots<k_q$ such that
  $
  \dfrac{G_1\cdots G_q}{x_{k_1}\cdots x_{k_q}}
  \det
  \begin{pmatrix}
     \ell_{1k_1}&\cdots&\ell_{1k_q}\\
     \vdots&&\vdots\\
     \ell_{qk_1}&\cdots&\ell_{qk_q}
    \end{pmatrix}
  \neq0.
  $
  It follows directly that the $q$ elements  $(\ell_{j1},\ldots,\ell_{jn})\in\Omega$ are independent, and therefore the dimension of $\mathit{Vect}_\KK\Omega$ is no less than $q$.

  Suppose $(\ell_1,\ldots,\ell_n)\in\Omega$, then it satisfies the first integral equations \eqref{eq:22}.
  We have integers $K_1,\ldots,K_p$ such that
  \begin{equation}
  \label{eq:31}
    \sum_{m=1}^n\ell_m\ln\mu_{im}=2K_i\sqrt{-1}\pi,\quad i=1,\ldots,p.
  \end{equation}

  If the system is infinitesimally integrable and the family of the infinitesimal generators $X_i=\sum_{m=1}^n\ln\mu_{im}x_m\dfrac{\partial}{\partial x_m}$ is weakly non-resonant, then all $K_i$ in equation \eqref{eq:31} vanish and we get linear equations
  \begin{equation}
  \label{eq:32}
    \sum_{m=1}^n\ell_m\ln\mu_{im}=0,\quad i=1,\ldots,p.
  \end{equation}
  It means that $\Omega$ is contained in the space of solutions of equations \eqref{eq:32}. By the definition of integrability,  we have the $p$ vectors $(\ln\mu_{i1},\ldots,\ln\mu_{in})$ are independent and therefore the space of solutions of \eqref{eq:32} is of dimension $n-p=q$ and therefore the dimension of $\mathit{Vect}_\KK\Omega$ is no more than $q$. Hence, under the assumption of weak non-resonance, the vector space $\mathit{Vect}_\KK\Omega$ is exactly the space of solutions of  \eqref{eq:32} over $\KK$.

  If the system is projectively hyperbolic, we consider the real parts on both sides of equation \eqref{eq:31} and get
  \begin{equation}
  \label{eq:33}
    \sum_{m=1}^n\ell_m\ln|\mu_{im}|=0,\quad i=1,\ldots,p.
  \end{equation}
  It means that $\Omega$ is contained in the space of solutions of equations \eqref{eq:33}. By the very definition of projective hyperbolicity,  the dimension of the space of solutions of \eqref{eq:33} is $n-p=q$ and therefore the dimension of $\mathit{Vect}_\KK\Omega$ is no more than $q$. Hence, under the assumption of projective hyperbolicity, the vector space $\mathit{Vect}_\KK\Omega$ is exactly the space of solutions of  \eqref{eq:33} over $\KK$.

\end{proof}

We remark that the dimension of $\mathit{Vect}_\KK\Omega$ could be bigger than $q$ without weak non-resonance.
for example, the linear diffeomorphism $\Phi(x,y)=({\sqrt{-1}x,-\sqrt{-1}y})$ on $\CC^2$ have monomial first integrals $x^4,\,y^4,\,xy$ and therefore the dimension of $\mathit{Vect}_\KK\Omega$ is $2$. In this case, the diffeomorphism is obviously weakly resonant.

%

\begin{proposition}
\label{prop:MonomialAreFI}

Let $(\Phi_1,\Phi_2,\ldots,\Phi_p,F_1,\ldots,F_q)$ be a formal non-degenerate integrable system in which the diffeomorphisms are in the Poincar\'e-Dulac normal form.  If the family of the linear parts of the diffeomorphisms is either projectively hyperbolic or infinitesimally integrable with a weakly non-resonant family of generators, then the common first integrals of the semi-simple parts of the diffeomorphisms are also first integrals of the (nonlinear) diffeomorphisms.
\end{proposition}
\begin{proof}
According to Lemma \ref{Lem:NFofFI}, formal or analytic first integrals of $\Phi_1,\Phi_2,\ldots,\Phi_p$ are formal or analytic first integrals of the semi-simple parts of the diffeomorphisms provided that the diffeomorphisms are in the Poincar\'e-Dulac normal form. Then any first integral is a series of finitely many monomial generators $G_1,\ldots,G_r$ which have exponents in $\Omega$. And the Lemma \ref{lem:VectOmega} shows that $\Omega$ lies in the $q$-dimensional vector space $\mathit{Vect}_\KK\Omega$ if the family of the linear parts of the diffeomorphisms is either projectively hyperbolic or infinitesimally integrable with a weakly non-resonant family of generators
%

Now we turn to formal integrable system $(\hat\Phi_1,\hat\Phi_2,\ldots,\hat\Phi_p,\hat F_1,\ldots,\hat F_q)$. We assume by Ziglin's lemma \cite{Ziglin1982} that the homogeneous parts  $F_1^{(low_1)},\ldots,F_q^{(low_q)}$ of lowest degree of $\hat F_1,\ldots,\hat F_q$ are functionally independent almost everywhere. For convenience, we use new first integrals $\hat F_1^{\frac{LCM}{low_1}},\ldots,\hat F_q^{\frac{LCM}{low_q}}$ where $LCM$ denotes the least common multiple of $low_1,\ldots,low_q$ so that their lowest degree are all the same. In the following, the new first integrals are still denoted by $\hat F_1,\ldots,\hat F_q$ and their homogeneous parts of lowest degree are denoted by $F_1^{(low)},\ldots,F_q^{(low)}$.

Let $H_k=x_1^{\ell_{k1}}\cdots x_n^{\ell_{kn}}$ for $k=1,\ldots,\tau$ be all monomial first integrals of $\Phi_i^{(1)}$ such that the coefficients are $1$ and $\ell_{k1}+\cdots+\ell_{kn}=low$.
Now write
\[
F_j^{(low)}=c_{j1}H_1+\cdots+c_{j\tau}H_\tau,
\]
where $c_{jk}$ are constants and the rank of the $q$ by $\tau$ matrix $C=(c_{jk})$ is $q\leqslant\tau$ by the functional independence.

As $\Phi_i=(\phi_{i1},\ldots,\phi_{in})$ is in the normal form, then for any monomial first integral $G=x_1^{\ell_1}\cdots x_n^{\ell_n}$ of $\Phi_i^{(1)}(x)=(\mu_{i1}x_1,\ldots,\mu_{in}x_n)$, we have $G=(\mu_{i1}x_1)^{\ell_1}\cdots(\mu_{in}x_n)^{\ell_n}$ and then
\begin{equation}
\label{eq:34}
 G\circ\Phi_i=\phi_{i1}^{\ell_1}\cdots\phi_{in}^{\ell_n}
 =G\left(1+\frac{\phi_{i1}^{(\geqslant2)}}{\mu_{i1}x_1}\right)^{\ell_1}\cdots\left(1+\frac{\phi_{in}^{(\geqslant2)}}{\mu_{in}x_n}\right)^{\ell_n},
\end{equation}
where $\phi_{im}^{(\geqslant2)}$ denotes the nonlinear part of $\phi_{im}$.
It is easy to see that the homogeneous part of degree $\ell_1+\cdots+\ell_n$ of $G\circ\Phi_i$ is $G$ and the homogeneous part of degree $\ell_1+\cdots+\ell_n+1$ is
\begin{equation}
\label{eq:35}
 (G\circ\Phi_i)^{(\ell_1+\cdots+\ell_n+1)}=G\left(\ell_1\dfrac{\phi_{i1}^{(2)}}{\mu_{i1}x_1}+\cdots+\ell_n\dfrac{\phi_{in}^{(2)}}{\mu_{in}x_n}\right).
\end{equation}

As $\hat F_j\circ\Phi_i=\hat F_j$, their homogeneous parts of degree $low+1$ must be the same, i.e., $(\hat F_j\circ\Phi_i)^{(low+1)}=F_j^{(low+1)}$. On the other hand, we have
\[
\begin{aligned}
 (\hat F_j\circ\Phi_i)^{(low+1)}&=(F_j^{(low)}\circ\Phi_i)^{(low+1)}+(F_j^{(low+1)}\circ\Phi_i)^{(low+1)}\\
 &=(F_j^{(low)}\circ\Phi_i)^{(low+1)}+(F_j^{(low+1)}\circ\Phi_i^{(1)})=(F_j^{(low)}\circ\Phi_i)^{(low+1)}+F_j^{(low+1)}.
\end{aligned}
\]
Then we get
\[
(F_j^{(low)}\circ\Phi_i)^{(low+1)}=0.
\]
Substitute $F_j^{(low)}$ by $c_{j1}H_1+\cdots+c_{j\tau}H_\tau$ and use the equation \eqref{eq:35}, we get
\[
 \sum_{m=1}^n \left(\sum_{k=1}^{\tau}c_{jk}\ell_{km}H_k\right)\dfrac{\phi_{im}^{(2)}}{{\mu_{im}x_m}}=0,
 \quad j=1,\ldots,q.
\]
Using matrices, the equations above are equivalent to
\begin{equation}
\label{eq:36}
  \begin{pmatrix}
    c_{11}H_1&\cdots& c_{1\tau}H_\tau\\
    \vdots&&\vdots\\
    c_{q1}H_1&\cdots& c_{q\tau}H_\tau
  \end{pmatrix}_{q\times\tau}
  \begin{pmatrix}
    \ell_{11}&\cdots&\ell_{1n}\\
    \vdots&&\vdots\\
    \ell_{\tau 1}&\cdots&\ell_{\tau n}
  \end{pmatrix}_{\tau\times n}
  \begin{pmatrix}
    \frac{\phi_{i1}^{(2)}}{\mu_{i1}x_1}\\
    \vdots\\
    \frac{\phi_{in}^{(2)}}{\mu_{in}x_n}
  \end{pmatrix}_{n\times1}
  =0.
\end{equation}

Assume that $H_1,\ldots,H_q$ are functionally independent almost everywhere and then for any $k$ in $\{1,\ldots,\tau\}$ we can write $H_k=H_1^{\alpha_{k1}}\cdots H_q^{\alpha_{kq}}$. Equivalently, the $q$ vectors $(\ell_{11},\ldots,\ell_{1n}),\ldots,(\ell_{q1},\ldots,\ell_{qn})$ are linearly independent and for any $k$ in $\{1,\ldots,\tau\}$ we have $(\ell_{k1},\ldots,\ell_{kn})=\sum_{j=1}^q\alpha_{kj}(\ell_{j1},\ldots,\ell_{jn})$. Write the $\tau$ by $q$ matrix $(\alpha_{kj})=
\begin{pmatrix}
   Id_q\\
   B
\end{pmatrix}
$
with $B$ the submatrix consisting of the last $\tau-q$ rows,
then we have
\[
\begin{aligned}
&\quad
  \begin{pmatrix}
    c_{11}H_1&\cdots& c_{1\tau}H_\tau\\
    \vdots&&\vdots\\
    c_{q1}H_1&\cdots& c_{q\tau}H_\tau
  \end{pmatrix}_{q\times\tau}
  \begin{pmatrix}
   Id_q&0\\
   B&Id_{\tau-q}
  \end{pmatrix}
  \begin{pmatrix}
   Id_q&0\\
   -B&Id_{\tau-q}
  \end{pmatrix}
  \begin{pmatrix}
    \ell_{11}&\cdots&\ell_{1n}\\
    \vdots&&\vdots\\
    \ell_{\tau 1}&\cdots&\ell_{\tau n}
  \end{pmatrix}_{\tau\times n}\\
&=
  \begin{pmatrix}
    \displaystyle\sum_{k=1}^\tau\alpha_{k1}c_{1k}H_k&\mkern-5mu\cdots\mkern-5mu&\displaystyle\sum_{k=1}^\tau\alpha_{kq}c_{1k}H_k&c_{1\,q+1}H_{q+1}&\mkern-5mu\cdots\mkern-5mu&c_{1\tau}H_\tau\\
    \vdots&&\vdots&\vdots&&\vdots\\
    \displaystyle\sum_{k=1}^\tau\alpha_{k1}c_{qk}H_k&\mkern-5mu\cdots\mkern-5mu&\displaystyle\sum_{k=1}^\tau\alpha_{kq}c_{qk}H_k&c_{q\,q+1}H_{q+1}&\mkern-5mu\cdots\mkern-5mu&c_{q\tau}H_\tau
  \end{pmatrix}_{q\times\tau}
  \begin{pmatrix}
    \ell_{11}&\cdots&\ell_{1n}\\
    \vdots&&\vdots\\
    \ell_{q1}&\cdots&\ell_{qn}\\
    0&\cdots&0\\
    \vdots&&\vdots\\
    0&\cdots&0
  \end{pmatrix}_{\tau\times n}
 \end{aligned}
\]
and we get from equation \eqref{eq:36} that
\begin{equation}
\label{eq:37}
  \begin{pmatrix}
    \displaystyle\sum_{k=1}^\tau\alpha_{k1}c_{1k}H_k&\mkern-5mu\cdots\mkern-5mu&\displaystyle\sum_{k=1}^\tau\alpha_{kq}c_{1k}H_k\\
    \vdots&&\vdots\\
    \displaystyle\sum_{k=1}^\tau\alpha_{k1}c_{qk}H_k&\mkern-5mu\cdots\mkern-5mu&\displaystyle\sum_{k=1}^\tau\alpha_{kq}c_{qk}H_k
  \end{pmatrix}_{q\times q}
  \mkern-25mu
  \begin{pmatrix}
    \ell_{11}&\cdots&\ell_{1n}\\
    \vdots&&\vdots\\
    \ell_{q1}&\cdots&\ell_{qn}
  \end{pmatrix}_{q\times n}
  \begin{pmatrix}
    \dfrac{\phi_{i1}^{(2)}}{\mu_{i1}x_1}\\
    \vdots\\
    \dfrac{\phi_{in}^{(2)}}{\mu_{in}x_n}
  \end{pmatrix}_{n\times1}
  \mkern-25mu
  =0.
\end{equation}

Now let us compute the explicit expression of $dF_q^{(low)}\wedge\cdots\wedge dF_q^{(low)}$.
\[
\begin{aligned}
  &\qquad dF_q^{(low)}\wedge\cdots\wedge dF_q^{(low)}\\
  &=\sum_{1\leqslant k_1<\cdots<k_q\leqslant\tau}\det
    \begin{pmatrix}
     c_{1\,k_1}&\cdots&c_{1\,k_q}\\
     \vdots&&\vdots\\
     c_{q\,k_1}&\cdots&c_{q\,k_q}
    \end{pmatrix} dH_{k_1}\wedge\cdots\wedge dH_{k_q}\\
  &=\sum_{1\leqslant k_1<\cdots<k_q\leqslant\tau}\mkern-20mu\det\mkern-5mu
    \begin{pmatrix}
     c_{1\,k_1}&\cdots&c_{1\,k_q}\\
     \vdots&&\vdots\\
     c_{q\,k_1}&\cdots&c_{q\,k_q}
    \end{pmatrix}
    \sum_{1\leqslant m_1<\cdots<m_q\leqslant n}\mkern-20mu\det\mkern-5mu
    \begin{pmatrix}
     \ell_{k_1 m_1}&\cdots&\ell_{k_1 m_q}\\
     \vdots&&\vdots\\
     \ell_{k_q m_1}&\cdots&\ell_{k_q m_q}
    \end{pmatrix}\dfrac{H_{k_1}\cdots H_{k_q}}{x_{m_1}\cdots x_{m_q}}dx_{m_1}\wedge\cdots\wedge dx_{m_q}\\
  &=\sum_{1\leqslant m_1<\cdots<m_q\leqslant n}
    \sum_{1\leqslant k_1<\cdots<k_q\leqslant\tau}\mkern-20mu\det
    \{
    \begin{pmatrix}
     c_{1\,k_1}H_{k_1}&\cdots&c_{1\,k_q}H_{k_q}\\
     \vdots&&\vdots\\
     c_{q\,k_1}H_{k_1}&\cdots&c_{q\,k_q}H_{k_q}
    \end{pmatrix}
    \begin{pmatrix}
     \ell_{k_1 m_1}&\cdots&\ell_{k_1 m_q}\\
     \vdots&&\vdots\\
     \ell_{k_q m_1}&\cdots&\ell_{k_q m_q}
    \end{pmatrix}
    \}\dfrac{dx_{m_1}\wedge\cdots\wedge dx_{m_q}}{x_{m_1}\cdots x_{m_q}}
\end{aligned}
\]
Remember $(\ell_{k\,m_1},\ldots,\ell_{k\,m_q})=\sum_{j=1}^{q}\alpha_{kj}(\ell_{j\,m_1},\ldots,\ell_{j\,m_q})$, we have
\[
 \begin{pmatrix}
     \ell_{k_1 m_1}&\cdots&\ell_{k_1 m_q}\\
     \vdots&&\vdots\\
     \ell_{k_q m_1}&\cdots&\ell_{k_q m_q}
 \end{pmatrix}
 =
 \begin{pmatrix}
     \alpha_{k_1\,1}&\cdots&\alpha_{k_1\,q}\\
     \vdots&&\vdots\\
     \alpha_{k_q\,1}&\cdots&\alpha_{k_q\,q}
 \end{pmatrix}
 \begin{pmatrix}
     \ell_{1\,m_1}&\cdots&\ell_{1\,m_q}\\
     \vdots&&\vdots\\
     \ell_{q\,m_1}&\cdots&\ell_{q\,m_q}
 \end{pmatrix},
\]
and therefore we can split the two summations on $1\leqslant m_1<\cdots<m_q\leqslant n$ and $1\leqslant k_1<\cdots<k_q\leqslant\tau$ in the expression of $dF_q^{(low)}\wedge\cdots\wedge dF_q^{(low)}$. Concretely, $dF_q^{(low)}\wedge\cdots\wedge dF_q^{(low)}$ is the product of
\[
\sum_{1\leqslant m_1<\cdots<m_q\leqslant n}
    \dfrac{1}{x_{m_1}\cdots x_{m_q}}
    \det\begin{pmatrix}
            \ell_{1\,m_1}&\cdots&\ell_{1\,m_q}\\
            \vdots&&\vdots\\
            \ell_{q\,m_1}&\cdots&\ell_{q\,m_q}
         \end{pmatrix}
    dx_{m_1}\wedge\cdots\wedge dx_{m_q}
\]
and the homogeneous polynomial function
\[
\sum_{1\leqslant k_1<\cdots<k_q\leqslant\tau}
    \det
    \{
    \begin{pmatrix}
      c_{1\,k_1}H_{k_1}&\cdots&c_{1\,k_q}H_{k_q}\\
      \vdots&&\vdots\\
      c_{q\,k_1}H_{k_1}&\cdots&c_{q\,k_q}H_{k_q}
    \end{pmatrix}
    \begin{pmatrix}
      \alpha_{k_1\,1}&\cdots&\alpha_{k_1\,q}\\
      \vdots&&\vdots\\
      \alpha_{k_q\,1}&\cdots&\alpha_{k_q\,q}
    \end{pmatrix}
    \}.
\]

This polynomial function cannot be zero since $dF_q^{(low)}\wedge\cdots\wedge dF_q^{(low)}\neq0$, and it  equals to the determinant of the leftmost matrix $M(H_1,\ldots,H_\tau)$ in equation \eqref{eq:37}.
In fact, as the determinant is a linear function of each column, we write $\det M(H_1,\ldots,H_\tau)$ as a sum over all $k_1,\ldots,k_q$ from $1$ to $\tau$ of $\tau^q$ determinants
\[
\det
    \begin{pmatrix}
      \alpha_{k_11}c_{1\,k_1}H_{k_1}&\cdots&\alpha_{k_qq}c_{1\,k_q}H_{k_q}\\
      \vdots&&\vdots\\
      \alpha_{k_11}c_{q\,k_1}H_{k_1}&\cdots&\alpha_{k_qq}c_{q\,k_q}H_{k_q}
    \end{pmatrix}
=\alpha_{k_11}\cdots\alpha_{k_qq}\det
    \begin{pmatrix}
      c_{1\,k_1}&\cdots&c_{1\,k_q}\\
      \vdots&&\vdots\\
      c_{q\,k_1}&\cdots&c_{q\,k_q}
    \end{pmatrix}
  H_{k_1}\cdots H_{k_q}
\]
which must vanish if two indices $k_j$ and $k_{j'}$ happen to be equal; fix $q$ pairwise distinct indices $\{k_1,\ldots,k_q\}$ in $\{1,\ldots,\tau\}$ and suppose $k_1<\cdots<k_q$ , then there are $q!$ terms similar to $H_{k_1}\cdots H_{k_q}$  and the sum of them is just
\[
 \begin{aligned}
    P(H_{k_1},\ldots,H_{k_q})
    &:=
    \sum_{\substack{\{k'_1,\ldots,k'_q\}\\ =\{k_1,\ldots,k_q\}}}
    \alpha_{k'_11}\cdots\alpha_{k'_qq}\det
    \begin{pmatrix}
      c_{1\,k'_1}&\cdots&c_{1\,k'_q}\\
      \vdots&&\vdots\\
      c_{q\,k'_1}&\cdots&c_{q\,k'_q}
    \end{pmatrix}
       H_{k_1}\cdots H_{k_q}\\
    &=
    \sum_{\substack{\{k'_1,\ldots,k'_q\}\\ =\{k_1,\ldots,k_q\}}}
    \alpha_{k'_11}\cdots\alpha_{k'_qq}
    \,\,\epsilon(k'_1,\ldots,k'_q)\det
    \begin{pmatrix}
      c_{1\,k_1}&\cdots&c_{1\,k_q}\\
      \vdots&&\vdots\\
      c_{q\,k_1}&\cdots&c_{q\,k_q}
    \end{pmatrix}
       H_{k_1}\cdots H_{k_q}\\
    &=\det
    \begin{pmatrix}
      \alpha_{k_1\,1}&\cdots&\alpha_{k_1\,q}\\
      \vdots&&\vdots\\
      \alpha_{k_q\,1}&\cdots&\alpha_{k_q\,q}
    \end{pmatrix}
    \det
    \begin{pmatrix}
      c_{1\,k_1}&\cdots&c_{1\,k_q}\\
      \vdots&&\vdots\\
      c_{q\,k_1}&\cdots&c_{q\,k_q}
    \end{pmatrix}
    H_{k_1}\cdots H_{k_q},
 \end{aligned}
\]
in which $\epsilon(k'_1,\ldots,k'_q)=\pm1$ is the sign of the permutation $(k'_1,\ldots,k'_q)\mapsto(k_1,\ldots,k_q)$. Hence $\det M(H_1,\ldots,H_\tau)=\sum_{1\leqslant k_1<\cdots<k_q\leqslant\tau}P(H_{k_1},\ldots,H_{k_q})$.

Back to equation \eqref{eq:37}, as the matrix $M(H_1,\ldots,H_\tau)$ is invertible almost everywhere, it follows that, almost everywhere, we have
\[
  \begin{pmatrix}
    \ell_{11}&\cdots&\ell_{1n}\\
    \vdots&&\vdots\\
    \ell_{q1}&\cdots&\ell_{qn}
  \end{pmatrix}_{q\times n}
  \begin{pmatrix}
    \frac{\phi_{i1}^{(2)}}{\mu_{i1}x_1}\\
    \vdots\\
    \frac{\phi_{in}^{(2)}}{\mu_{in}x_n}
  \end{pmatrix}_{n\times1}
  =0.
\]
It follows by Lemma \ref{lem:VectOmega} that, as polynomial functions,
\begin{equation}
\label{eq:38}
G\left(\ell_1 \frac{\phi_{i1}^{(2)}}{\mu_{i1}x_1}+\cdots+\ell_n \frac{\phi_{in}^{(2)}}{\mu_{in}x_n}\right)=0, \quad \forall G=x_1^{\ell_1}\cdots x_n^{\ell_n} \mbox{~with~} (\ell_1,\ldots,\ell_n)\in\Omega.
\end{equation}
In other words, the homogeneous part of degree $\ell_1+\cdots+\ell_n+1$ of $G\circ\Phi_i$ vanishes for any common monomial first integral $G=x_1^{\ell_1}\cdots x_n^{\ell_n}$ of $\Phi_i^{(1)}$ by equation \eqref{eq:35}.

We will show by induction that the homogeneous part of $G\circ\Phi_i$ with degree larger than $\ell_1+\cdots+\ell_n$ also vanishes for  $G=x_1^{\ell_1}\cdots x_n^{\ell_n}$, then we can say that any monomial first integral of $\Phi_i^{(1)}$ is a first integral of $\Phi_i$. Assume the statement is true up to degree $\ell_1+\cdots+\ell_n+\sigma$, $\sigma>0$; it  follows naturally that for any homogeneous polynomial first integral $F^{(\ell)}$, the homogeneous parts up to degree $\ell+\sigma$ of $F^{(\ell)}\circ\Phi_i$ all vanish.

Let $\xi_m=\ln(1+\dfrac{\phi_{im}^{(\geqslant2)}}{\mu_{im}x_m})$ and $\eta=\ln\left((1+\frac{\phi_{i1}^{(\geqslant2)}}{\mu_{i1}x_1})^{\ell_1}\cdots(1+\frac{\phi_{in}^{(\geqslant2)}}{\mu_{in}x_n})^{\ell_n}\right)
=\ell_1\xi_1+\cdots+\ell_n\xi_n$. Use the convention that the degree with respect to $x$ of $\dfrac{\phi_{im}^{(s)}}{\mu_{im}x_m}$ is $s-1$ and rewrite $\eta=\eta^{(1)}+\eta^{(2)}+\cdots$ where $\eta^{(s)}$ denotes the homogeneous part of degree $s$ with respect to $x$. Rewrite equation \eqref{eq:34} as $G\circ\Phi_i=Ge^\eta=G(1+\eta+\frac{1}{2}\eta^2+\cdots)$ for those $\eta$ with $|\eta|<\infty$, that is when $x$ does not belong to the union of hyperplane coordinates.
By our assumption, we get that every homogeneous part of degree no more than $\sigma$ in $(\eta+\frac{1}{2}\eta^2+\cdots)$ must vanish. Then we have $\eta^{(1)}=\eta^{(2)}=\cdots=\eta^{(\sigma)}=0$ because for any $s$, we have
\[
 (\eta+\frac{1}{2}\eta^2+\cdots)^{(s)}=\sum_{t=1}^s\sum_{s_1+\cdots+s_t=s}c_{s_1\cdots s_t}\eta^{(s_1)}\cdots \eta^{(s_t)},
\]
in which $c_{s_1\cdots s_t}$ are constants; and therefore the degree of the first possibly nonvanishing homogeneous part of $(\eta+\frac{1}{2}\eta^2+\cdots)$ must larger than $\sigma$.

For degree $\sigma+1$, we have
\[
 (\eta+\frac{1}{2}\eta^2+\cdots)^{(\sigma+1)}=\sum_{t=1}^{\sigma+1}\sum_{s_1+\cdots+s_t=\sigma+1}c_{s_1\cdots s_t}\eta^{(s_1)}\cdots \eta^{(s_t)}=\eta^{(\sigma+1)}.
\]
We get that the homogeneous part of degree $\ell_1+\cdots+\ell_n+\sigma+1$ of $G\circ\Phi_i$ is just $G\eta^{(\sigma+1)}$, which reads
\begin{equation}
\label{eq:39}
\begin{aligned}
  (G\circ\Phi_i)^{(\ell_1+\cdots+\ell_n+\sigma+1)}&=G(\ell_1\left(\ln(1+\frac{\phi_{i1}^{(\geqslant2)}}{\mu_{i1}x_1})\right)^{(\sigma+1)}+\cdots+\ell_n\left(\ln(1+\frac{\phi_{in}^{(\geqslant2)}}{\mu_{in}x_n})\right)^{(\sigma+1)})\\
  &=G(\ell_1\xi_1^{(\sigma+1)}+\cdots+\ell_n\xi_n^{(\sigma+1)}).
\end{aligned}
\end{equation}

Now consider the homogeneous part $(\hat F_j\circ\Phi_i)^{(low+\sigma+1)}$ of degree $low+\sigma+1$ of $\hat F_j\circ\Phi_i$, which is
\begin{equation}
\label{eq:310}
(F_j^{(low)}\circ\Phi_i)^{(low+\sigma+1)}+\sum_{s=1}^{\sigma}(F_j^{(low+s)}\circ\Phi_i)^{(low+\sigma+1)}+(F_j^{(low+\sigma+1)}\circ\Phi_i)^{(low+\sigma+1)}.
\end{equation}
We recall that, by \ref{Lem:NFofFI}, $F_j^{(low+s)}$ is a common homogeneous polynomial first integral of the $\Phi^{(1)}_j$'s. By our inductive hypothesis, the $\sigma$ terms $(F_j^{(low+s)}\circ\Phi_i)^{(low+\sigma+1)}$ in the middle of equation \eqref{eq:310} vanish; the last term $(F_j^{(low+\sigma+1)}\circ\Phi_i)^{(low+\sigma+1)}$ is just $(F_j^{(low+\sigma+1)}\circ\Phi_i^{(1)})=F_j^{(low+\sigma+1)}$. Hence we get from $(\hat F_j\circ\Phi_i)^{(low+\sigma+1)}=F_j^{(low+\sigma+1)}$ that the first term in equation \eqref{eq:310} vanishes, i.e.,
\[
 (F_j^{(low)}\circ\Phi_i)^{(low+\sigma+1)}=0.
\]
Substitute $F_j^{(low)}$ by $c_{j1}H_1+\cdots+c_{j\tau}H_\tau$ and use equation \eqref{eq:39}, we have
\[
\begin{aligned}
 &c_{j1}(H_1\circ\Phi_i)^{(low+\sigma+1)}+\cdots+c_{j\tau}(H_\tau\circ\Phi_i)^{(low+\sigma+1)}\\
 =&\,c_{j1}H_1(\ell_{11}\xi_1^{(\sigma+1)}+\cdots+\ell_{1n}\xi_n^{(\sigma+1)})+\cdots+c_{j\tau}H_\tau(\ell_{\tau 1}\xi_1^{(\sigma+1)}+\cdots+\ell_{\tau n}\xi_n^{(\sigma+1)})=0,
\end{aligned}
\]
that is,
\[
 \sum_{m=1}^n \left(\sum_{k=1}^{\tau}c_{jk}\ell_{km}H_k\right)\xi_m^{(\sigma+1)}=0,
 \quad j=1,\ldots,q.
\]

Using matrix and similar to equation \eqref{eq:36}, the above equations
\begin{equation}
\label{eq:311}
  \begin{pmatrix}
    c_{11}H_1&\cdots& c_{1\tau}H_\tau\\
    \vdots&&\vdots\\
    c_{q1}H_1&\cdots& c_{q\tau}H_\tau
  \end{pmatrix}_{q\times\tau}
  \begin{pmatrix}
    \ell_{11}&\cdots&\ell_{1n}\\
    \vdots&&\vdots\\
    \ell_{\tau 1}&\cdots&\ell_{\tau n}
  \end{pmatrix}_{\tau\times n}
  \begin{pmatrix}
    \xi_1^{(\sigma+1)}\\
    \vdots\\
    \xi_n^{(\sigma+1)}
  \end{pmatrix}_{n\times1}
  =0.
\end{equation}
Apply the same argument  from equation \eqref{eq:36} to equation \eqref{eq:38}, we can get by equation \eqref{eq:311} that
\begin{equation}
\label{eq:312}
 \ell_1\xi_1^{(\sigma+1)}+\cdots+\ell_n\xi_n^{(\sigma+1)}=0, \quad \mbox{~for all~} (\ell_1,\ldots,\ell_n)\in\Omega.
\end{equation}
Take equation \eqref{eq:312} back to equation \eqref{eq:39}, we get that the homogeneous part of degree $\ell_1+\cdots+\ell_n+\sigma+1$ of $G\circ\Phi_i$ vanishes. We finish our inductive step.


\end{proof}

\color{black}

\begin{lemma}[Division Lemma]
\label{lem:Division}
  Let $(\Phi_1,\Phi_2,\ldots,\Phi_p,F_1,\ldots,F_q)$ be a non-degenerate integrable system of type $(p,q)$ such that the diffeomorphisms are in Poincar\'e-Dulac normal form. Write $\Phi_i=(\phi_{i1},\ldots,\phi_{in})$ for $i=1,\ldots,p$. If the family \{$\Phi_i^{(1)}\}$ is either projectively hyperbolic or infinitesimally integrable with a weakly non-resonant family of generators, then we have $\phi_{im}$ is divisible by $x_m$ for $m=1,\ldots,n.$
\end{lemma}
\begin{proof}
  There are two cases according to different positions of the vector space $\mathit{Vect}_\KK\Omega$.

  Case 1: the vector space $\mathit{Vect}_\KK\Omega$ is not contained in any hyperplane.  In this case, for any $m$,  there exists an element $(\ell_1,\ldots,\ell_n)\in\Omega$ such that $\ell_m\neq0$. The equation $\prod_{k=1}^n \phi_{ik}^{\ell_k}=x_1^{\ell_1}\cdots x_n^{\ell_n}$ implies that $\prod_{k=1}^n \phi_{ik}^{\ell_k}$ is divisible by $x_m^{\ell_m}$. On the other hand, as the linear part of $\phi_{ik}$ is $\mu_{ik}x_k$, we get $\prod_{k\neq m}\phi_{ik}$ is not divisible by $x_m$ since its homogeneous part of lowest degree is $\prod_{k\neq m}\mu_{ik}x_k$. Hence, $\prod_{k\neq m} \phi_{ik}^{\ell_k}$ is not divisible by $x_m$ neither. Hence $\phi_{im}$ is divisible by $x_m$.

  Case 2: the vector space $\mathit{Vect}_\KK\Omega$ is contained in a hyperplane. Assume for any $(\ell_1,\ldots,\ell_n)\in\Omega$ we have $\ell_m=0$ and and $x_1^{\gamma_1}\cdots x_n^{\gamma_n}$ is a term of $\phi_{im}$ with $\gamma_m=0$. we have the indices $(\gamma_1,\ldots,\gamma_n)$ lie in $\mathcal R_m$ which satisfy the resonant equations \eqref{eq:21}. Then We have integers $K_1,\ldots,K_n$ such that
  \begin{equation}
  \label{eq:313}
    \ln\mu_{im}=\sum_{k=1}^n\gamma_k\ln\mu_{ik}+2K_i\pi\sqrt{-1},\quad i=1,\ldots,p.
  \end{equation}

  If the system is infinitesimally integrable and the family of the infinitesimal generators  $X_i=\sum_{m=1}^n\ln\mu_{im}x_m\frac{\partial}{\partial x_m}$ is weakly non-resonant, then we have $K_i=0$ for all $i$ and $$(\gamma_1,\ldots,\gamma_{m-1},-1,\gamma_{m+1},\ldots,\gamma_n)$$ is an integer solution of the equations
  \begin{equation}
  \label{eq:314}
    \sum_{k=1}^n\gamma_k\ln\mu_{ik}=0,\quad i=1,\ldots,p.
  \end{equation}
  As its $m$-th component is nonzero and therefore it cannot be expressed by a  linear combination of elements in $\Omega$, then we can get $q+1$ independent solutions of \eqref{eq:314} which contradicts with that the dimension of $\mathit{Vect}_\KK\Omega$ equals to $q$.

  If the system is projectively hyperbolic, then we consider the real parts on both sides of equation \eqref{eq:313}
   \[
   \ln|\mu_{im}|-\sum_{k=1}^n \gamma_k\ln|\mu_{ik}|=0,\quad i=1,\ldots,p.
   \]
   We can see that $(\gamma_1,\ldots,\gamma_{m-1},-1,\gamma_{m+1},\ldots,\gamma_n)$ is an integer solution of the equations
   \begin{equation}
   \label{eq:315}
     \sum_{k=1}^n \gamma_k\ln|\mu_{ik}|=0,\quad i=1,\ldots,p.
   \end{equation}
   Then the dimension of solutions of \eqref{eq:315} is larger than $q$ and so is that of $\mathit{Vect}_\KK\Omega$, which contradicts with that the dimension of $\mathit{Vect}_\KK\Omega$ equals to $q$.

   Hence, under the assumption of weak non-resonance or projective hyperbolicity, we have for every term $x_1^{\gamma_1}\cdots x_n^{\gamma_n}$ of $\phi_{im}$, its $m$-th exponent $\gamma_m>0$.
\end{proof}

We point out that our hypothesis is necessary.
\begin{example}
\label{ex:34}
Consider two commuting diffeomorphisms on $(\CC^2,0)$
\[
 \Phi_1(x,y)=(2x,4y+x^2)\quad\text{and}\quad\Phi_2(x,y)=(-3x,9y).
\]
The commuting diffeomorphisms are in the Poincar\'e-Dulac normal forms but they can not be put into normal forms stated in theorem \ref{thm:FormalNormalForm}. In this case, the integrable system without common first integrals is neither weakly non-resonant nor projectively hyperbolic.
\end{example}

We also note that if  $\Omega$ admits, say, only the first $p'$ entries are nonzero, which means the last $n-p'$ elements $\ell_{p'+1},\ldots,\ell_n$ must be zero, then the first integrals are independent of $x_{p'+1},\ldots,x_n$ by lemma \ref{Lem:NFofFI}. Moreover, we have all $\phi_{im}$ with $m\leqslant p'$ are independent of $x_{p'+1},\ldots,x_n$. In fact, we just proved that such $\phi_{im}$ is divisible by $x_m$, then by equation \eqref{eq:21}, the indices of the quotients of monomial terms in $\phi_{im}$ and $x_m$ also lie in $\Omega$, hence, the last $n-p'$ indices of every monomial term  in $\phi_{im}$ must be zero. Hence, consider projections of $\Phi_1,\ldots,\Phi_n$ to the plane of first $p'$ coordinates, then any $p'$ independent of them as diffeomorphisms on the coordinate plane together with the $q$ first integrals as functions on the coordinate plane form an integrable system of type $(p',q)$.

\noindent\textbf{End of the proof of theorem \ref{thm:FormalNormalForm}}

By the division lemma \ref{lem:Division}, there exist functions $\varphi_{im}$ such that $\phi_{im}=\mu_{im}x_m(1+\hat\varphi_{im})$ for all $i$ and all $m$. By proposition \ref{prop:MonomialAreFI}, we have $(x_1^{\gamma_1}\cdots x_n^{\gamma_n})\circ\Phi_i= x_1^{\gamma_1}\cdots x_n^{\gamma_n}$ for every $(\gamma_1,\ldots,\gamma_n)$ in $\Omega$, and after substitutions of $\phi_{im}$ and a reduction,  we get $\prod_{k=1}^n (1+\varphi_{ik})^{\gamma_k}=1$.

Notice the relation between $\mathcal R_m$ and $\Omega$ given by \eqref{eq:21} and \eqref{eq:22} respectively, every term of $\phi_{im}$ whose indices lie in $\mathcal R_m$ is a product of $x_m$ and a term of $\varphi_{im}$ whose indices lie in $\Omega$, so $\varphi_{im}$ are first integrals of $\Phi_j^{ss}$.

\section{Cases in analytic and smooth category}
\subsection*{\textit{Analytic case}}
For analytic integrable diffeomorphisms, we pay attention to the systems of the Poincar\'e type.
\begin{definition}(\cite{Gong-Stolovitch2016}[Definition 4.11])
  Let $\Phi_1,\ldots,\Phi_p$ be $p$ commuting diffeomorphisms and $(\mu_{i1},\ldots,\mu_{in})$ be the eigenvalues of the linear part of $\Phi_i$. We say that the family of the diffeomorphisms (or their linear part) is of \textbf{the Poincar\'e type} if there exist $d>1$ and $c>0$ such that, for each
  $(s_1,\ldots,s_n)\not\in\mathcal{R}_m$,
  there exists $\left(i', (s'_1,\ldots,s'_n)\right) \in\{1, \ldots, p\} \times \NN^n$ such that $\mu_{i1}^{s'_1}\cdots\mu_{in}^{s'_n}=\mu_{i1}^{s_1}\cdots\mu_{in}^{s_n}$ for all $1 \leq i \leq p, \mu_{i1}^{s'_1}\cdots\mu_{in}^{s'_n}-\mu_{i' m} \neq 0$, and
  $$
   \max \left(\left|\mu_{i'1}^{s'_1}\cdots\mu_{i'n}^{s'_n}\right|,\left|\mu_{i'1}^{s'_1}\cdots\mu_{i'n}^{s'_n}\right|^{-1}\right)>c^{-1} d^{s'_1+\cdots+s'_n},\,\, (s'_1-s_1,\ldots,s'_n-s_n) \in \NN^{n} \cup\left(-\NN^{n}\right).
  $$

\end{definition}

By a theorem of X. Gong and L. Stolovitch \cite{Gong-Stolovitch2016}[Theorem 4.13], which says that if a commutative family of finitely many germs of biholomorphisms of the Poincar\'e type is formally conjugate to the normal form \eqref{good-nf} satisfying \eqref{intgnf}, then it is holomorphically conjugate to the normal form, we get the following theorem.

\begin{theorem}
  Let $(\Phi_1,\ldots,\Phi_p,F_1,\ldots,F_q)$ be a non-degenerate analytic integrable system of type $(p,q)$ on $\KK^n$ around $0$ satisfying the condition in theorem \ref{thm:FormalNormalForm}. If the family of diffeomorphisms is of Poincar\'e type, then the system is analytically conjugate to the normal form \ref{good-nf} together with \ref{intgnf} as in theorem \ref{thm:FormalNormalForm}., i.e., the normalization is convergent.
\end{theorem}

We remark that for integrable systems of type $(1,n-1)$, any diffeomorphism satisfying the assumption that at least one eigenvalue does not lie on the unit circle in \cite{Zhang2013} is projectively hyperbolic and of the Poincar\'e type~:
\begin{proposition}
  Let $\Phi$ be an integrable diffeomorphism on $\KK^n$ in the Poincar\'e-Dulac normal form formally. Suppose its linear part is diagonal written as $\Phi^{(1)}(x)=(\mu_1x_1,\ldots,\mu_nx_n)$ and at least one of its eigenvalues does not lie on the unit circle, then $\Phi$ is of the Poincar\'e type.
\end{proposition}
\begin{proof}
  Suppose $x_1^{\ell_{j1}}\cdots x_n^{\ell_{jn}}, j=1,\ldots,n-1$ are $n-1$ independent first integrals of $\Phi^{(1)}$. Then then equation \eqref{eq:33} in this particular case becomes
  \begin{equation}
  \label{eq:40}
  L
    \begin{pmatrix}
        \ln|\mu_1|\\
        \vdots\\
        \ln|\mu_n|
    \end{pmatrix}
    :=
    \begin{pmatrix}
        \ell_{11}&\cdots&\ell_{1\,n}\\
        \vdots&&\vdots\\
        \ell_{n-1\,1}&\cdots&\ell_{n-1\,n}
    \end{pmatrix}
    \begin{pmatrix}
        \ln|\mu_1|\\
        \vdots\\
        \ln|\mu_n|
    \end{pmatrix}
    =0.
  \end{equation}
  As the $n-1$ by $n$ matrix $L$ has rank $n-1$ by independence, the dimension of the space of its solutions is one. Then the hypothesis that there exists at least one of the eigenvalues does not lie on the unit circle implies that $(\ln|\mu_1|,\ldots,\ln|\mu_n|)$ is a nonzero solution of equation \eqref{eq:40}; on the other hand, as $L:=(\ell_{ji})_{(n-1)\times n}$ is an integer matrix, equation \eqref{eq:40} has integer solutions. Thus there exists an integer solution $(k_1,\ldots,k_n)$ and a real number $c>0$ such that $(\ln|\mu_1|,\ldots,\ln|\mu_n|)=c(k_1,\ldots,k_n)$. Then we get $\ln|\mu_i|=ck_i$ and then $|\mu_i|=(e^c)^{k_i}=e^{ck_i}$.

  Now write $\mu_i=e^{ck_i}e^{\sqrt{-1}\Arg \mu_i}$ where $0\leqslant\Arg \mu_i<2\pi$ denotes the principal value of the argument of $\mu_i$, then by  the property $\mu_1^{\ell_{j1}}\cdots\mu_n^{\ell_{jn}}=1$, there exist integers $K_1,\ldots,K_{n-1}$ such that
  \begin{equation}
  \label{eq:41}
   \ell_{j1}\Arg \mu_1+\cdots+\ell_{jn}\Arg \mu_n=2K_j\pi,\quad j=1,\ldots,n-1.
  \end{equation}
  This is a (non-homogeneous if $K_j\neq 0$) linear system and its real solutions form a one dimensional affine space: the difference of any two solutions is a solution of \eqref{eq:40}.  Then, by the same argument as above, we can take a special solution $2\pi(\theta_1,\ldots,\theta_n)$ such that $\theta_i$'s are rational numbers and therefore there exists a real number $c'$ such that $(\Arg \mu_1,\ldots,\Arg \mu_n)=2\pi(\theta_1,\ldots,\theta_n)+c'(k_1,\ldots,k_n)$.
  \[
   \mu_i=e^{ck_i}e^{\sqrt{-1}\,2\pi\theta_i}e^{\sqrt{-1}\,c'k_i}
   =e^{(c+\sqrt{-1}c')k_i}e^{\sqrt{-1}\,2\pi\theta_i}=d^{k_i}e^{\sqrt{-1}\,2\pi\theta_i},
  \]
  in which $d=e^{c+\sqrt{-1}c'}$ with $|d|=e^c>1$.

  For any $\mu_i$ with $|\mu_i|=1$ or equivalently $k_i=0$, $\mu_i=e^{\sqrt{-1}\,2\pi\theta_i}$. Then there exists a natural number $\alpha_i$ such that $\mu_i^{\alpha_i}=1$ since $\theta_i$ is rational and therefore $x_i^{\alpha_i}$ is a first integral of $\Phi^{ss}$.
  For any pair $\mu_i$ and $\mu_j$ with $|\mu_i|<1$ and $|\mu_j|>1$, we have $k_i<0<k_j$ and therefore there exist a pair of natural numbers $\beta_i$ and $\beta_j$ such that $\beta_ik_i+\beta_jk_j=0$ and $\beta_i\theta_i+\beta_j\theta_j\in\ZZ$. Then  $\mu_i^{\beta_i}\mu_j^{\beta_j}=1$ and therefore $x_i^{\beta_i}x_j^{\beta_j}$ is a first integral of $\Phi^{ss}$.

  We now claim that for any $(s_1,\ldots,s_n)\in\NN^n$, there exists $(s'_1,\ldots,s'_n)\in\NN^n$ such that
  \begin{itemize}
    \item $\mu_{1}^{s'_1}\cdots\mu_{n}^{s'_n}=\mu_{1}^{s_1}\cdots\mu_{n}^{s_n}$;
    \item either $\{s'_i: i \mbox{~satisfies~}|\mu_i|\leqslant1\}$ or $\{s'_j: j \mbox{~satisfies~}|\mu_j|\geqslant1\}$ is bounded.
  \end{itemize}
  In fact, let $M$ be a natural number bigger than all possible $\alpha_i,\beta_i,\beta_j$, then for $s_i$ with $i$ satisfies $|\mu_i|=1$, set $s'_i$ to be the remainder of $s_i$ divided by $\alpha_i$. In the same spirit, for $s_i>M$ and $s_j>M$ with $i\in I:=\{i:|\mu_i|<1\}$ and $j\in J:=\{j:|\mu_j|>1\}$, we take the maximal integer $m$ such that $(s_i,s_j)-m(\beta_j,\beta_j)$ is nonnegative and take this vector to replace $(s_i,s_j)$, then the new $s_i$ and $s_j$ satisfy $s_i<\beta_i<M$ or $s_j<\beta_j<M$; continue the operation for the other $(s_i,s_j)$'s with $i\in I, j\in J$ and $s_i>M, s_j>M$ and obviously the operation will stop in finite steps; set $s'_i$ with $i\in I\cup J$ to be the final $s_i$ after reductions. Then $(s'_1,\ldots,s'_n)$ satisfies the second request. It satisfies the first request since each operation holds the property.

  Now assume $s'_i<M$ for all $i\in\{i:|\mu_i|\leqslant1\}$. Remember $|\mu_j|\geqslant d>1$ for $j\in J$, we have
  \[
  \begin{aligned}
   \left|\mu_1^{s'_1}\cdots\mu_n^{s'_n}\right|
   &=\prod_{i\in\{i:|\mu_i|\leqslant1\}}|\mu_i|^{s'_i}\prod_{j\in J}|\mu_j|^{s'_j}\\
   &\geqslant\prod_{i\in\{i:|\mu_i|\leqslant1\}}|\mu_i|^{s'_i}d^{\sum_{j\in J}s'_j}\\
   &=\prod_{i\in\{i:|\mu_i|\leqslant1\}}\left({\frac{1}{d}|\mu_i|}\right)^{s'_i}d^{s'_1+\cdots+s'_n}
   \geqslant\prod_{i\in\{i:|\mu_i|\leqslant1\}}\left({\frac{1}{d}|\mu_i|}\right)^{M}d^{s'_1+\cdots+s'_n}.
  \end{aligned}
  \]
  Hence $\Phi$ is of the Poincar\'e type. One can get the same conclusion by a similar estimate on $\left|\mu_1^{s'_1}\cdots\mu_n^{s'_n}\right|^{-1}$ if $s'_j<M$ for all $j\in\{j:|\mu_j|\geqslant1\}$.
\end{proof}

With the help of a lemma (Lemma 2.5 in \cite{Zhang2013}) which claim that the linear part of the integrable diffeomorphism on $(\CC^n,0)$ of type $(1,n-1)$ is diagonalizable, it follows that
\begin{corollary}\cite{Zhang2013}
  An analytic integrable diffeomorphism of type $(1,n-1)$ on $(\CC^n,0)$ such that at least one of its eigenvalues does not lie on the unit circle is analytically conjugate to the normal form \ref{good-nf} together with \ref{intgnf} as in theorem \ref{thm:FormalNormalForm}.
\end{corollary}


\subsection*{\textit{Smooth case}}
In the smooth category, we only consider the weakly hyperbolic systems, which were firstly introduced and studied by M. Chaperon.

\begin{definition}([section1.2] in \cite{Chaperon2013})
\label{def:44}
  Let $\Phi_1,\ldots,\Phi_p$ be $p$ commuting diffeomorphisms on $(\KK^n,0)$. Suppose the eigenvalues of the semi-simple part $\Phi_i^{ss}$ of the linear part of $\Phi_i$ are $\mu_{i1},\ldots,\mu_{in}$.  For any $k\in\{1,\ldots,n\}$, we can get a linear form $c_k$ in $(\RR^p)^*$ defined by mapping $(t_1,\ldots,t_p)\in\RR^p$ to $\sum_{i=1}^p \ln|\mu_{ik}| t_i$.
  The $\ZZ^p$-action generated by the diffeomorphisms is called
\begin{itemize}
  \item\textbf{hyperbolic} if any $p$ linear forms in $\{c_1,\ldots,c_n\}$ are linearly independent in $(\RR^p)^*$;
  \item\textbf{weakly hyperbolic} if  the convex hull of any $p$ linear forms in $\{c_1,\ldots,c_n\}$ does not contain the origin of $(\RR^p)^*$.
\end{itemize}
Obviously, hyperbolicity implies weak hyperbolicity.
\end{definition}

We remark if $\KK=\CC$ and the diffeomorphisms are viewed as real diffeomorphisms from $(\RR^2)^n$ to itself, then the eigenvalues of $\Phi_i^{ss}$ are $\mu_{i1},\bar\mu_{i1},\ldots,\mu_{in},\bar\mu_{in}$. Then we can get $2n$ linear forms $c_k, k=1,2,\ldots,2n$ with $c_{2k}=c_{2k-1}, k=1,\ldots,n,$ and therefore the property that the convex hull of any $p$ linear forms in $\{c_1,\ldots,c_{2n}\}$ does not contain the origin of $(\RR^p)^*$ coincides with the previous one.

\begin{theorem}
 Let $(\Phi_1=\Phi,\Phi_2,\ldots,\Phi_p,F_1,\ldots,F_q)$ be a non-degenerate smooth integrable system of type $(p,q)$ on $\KK^n$ around $0$ satisfying the condition of theorem \ref{thm:FormalNormalForm}. If the system is  weakly hyperbolic, then the diffeomorphisms are smoothly conjugate to a smooth normal form of the form \ref{good-nf} together with \ref{intgnf} as in theorem \ref{thm:FormalNormalForm}.
\end{theorem}
\begin{proof}
   The idea of the proof is to construct another smooth integrable system which is formally conjugate to the original system and then we can apply Chaperon's theorem \cite{Chaperon2013}, which asserts that two weakly hyperbolic smooth $\ZZ^k\times\RR^m$-action germs are smoothly conjugate if and only if they are formally conjugate. 

   By theorem \ref{thm:FormalNormalForm}, the system is formally conjugate to
   \[
    \hat\Phi_i=(\mu_{i1}x_1(1+\hat\varphi_{i1}),\ldots,\mu_{in}x_n(1+\hat\varphi_{in})),\,i=1,\ldots,p,
  \]
  where $\hat\varphi_{ik}$'s are formal series of finitely many  generators, say $G_1,G_2,\ldots,G_r$, which are monomial first integrals of $\Phi_i^{ss}$'s. Moreover, these formal series satisfy the first integral relations $\prod_{k=1}^n (1+\hat\varphi_{ik})^{\gamma_k}=1$ in the formal sense for all $(\gamma_1,\ldots,\gamma_n)$ in the set $\Omega$ of common solutions of resonance equations \eqref{eq:22}.

  By the Borel's theorem, there exist smooth functions $\tilde\varphi_{ik}$'s which are indeed smooth functions of $G_1,G_2,\ldots,G_r$ whose formal Taylor power series expansion at the origin are just $\hat\varphi_{ik}$'s respectively. 
  Define
  \[
    \tilde\Phi_i=(\mu_{i1}x_1(1+\tilde\varphi_{i1}),\ldots,\mu_{in}x_n(1+\tilde\varphi_{in})),\,i=1,\ldots,p,
  \]
  A priori, this new family of smooth dffeomorphisms do not commute any longer. In order to retrieve the commutativity property, it is sufficient to replace functions $\tilde\varphi_{ik}$'s by smooth functions $\varphi_{ik}$'s such that $\varphi_{ik}$'s satisfy  $\prod_{k=1}^n (1+\varphi_{ik})^{\gamma_k}=1$ for $(\gamma_1,\ldots,\gamma_n)\in\Omega$. This replacement can by realized by only adjusting the flat parts of $\tilde\varphi_{ik}$'s as follows.

  Take $q=n-p$ $\QQ$-linearly independent elements in $\Omega$, denoted by $\omega_j:=(\omega_{j1},\ldots,\omega_{jn})$, $j=1,2,\ldots,q$. Fix $i$ and assume $\prod_{k=1}^n (1+\tilde\varphi_{ik})^{\omega_{jk}}=1+flat_{ij}$ for $j=1,\ldots,q$ in which $flat_{ij}$'s are flat functions i.e., their infinite jets at $0$ are zero. Take logarithm of the equations, we get
  \begin{equation}
  \label{eq:42}
    \omega_{j1}\ln(1+\tilde\varphi_{i1})+\cdots+\omega_{jn}\ln(1+\tilde\varphi_{in})=\ln(1+flat_{ij}),\quad j=1,2,\ldots,q.
  \end{equation}
  Remember the functions $\varphi_{ik}$ we are searching for satisfy
  \begin{equation}
  \label{eq:43}
    \omega_{j1}\ln(1+\varphi_{i1})+\cdots+\omega_{jn}\ln(1+\varphi_{in})=0,\quad j=1,2,\ldots,q.
  \end{equation}
  Assume without loss of generality the first $q$ columns of the matrix $(\omega_{jk})$ are independent and let $\varphi_{ik}=\tilde\varphi_{ik}$ for $k=q+1,\ldots,n$. For $k=1,\ldots,q$, let $\varphi_{ik}$ be the unique solution the linear equations obtained by \eqref{eq:42} minus \eqref{eq:43}
  \[
    \sum_{m=1}^q \omega_{jm}\left(\ln(1+\tilde\varphi_{im})-\ln(1+\varphi_{im})\right)=\ln(1+flat_{ij}),\quad j=1,2,\ldots,q.
  \]
  We get immediately that $\prod_{k=1}^n (1+\varphi_{ik})^{\omega_{jk}}=1$ for $j=1,\ldots,q$, and it is also easy to verify $\varphi_{i\ell}-\tilde\varphi_{i\ell}$ is flat since $\ln{\dfrac{1+\tilde\varphi_{i\ell}}{1+\varphi_{i\ell}}}$ is flat.
  By Lemma \ref{lem:VectOmega}, any element $(\gamma_1,\ldots,\gamma_n)$ in $\Omega$ is a $\QQ$-linear combination of $\omega_1,\ldots,\omega_q$. Hence, we get $\prod_{k=1}^n (1+\varphi_{ik})^{\gamma_k}=1$  for all $(\gamma_1,\ldots,\gamma_n)$ in $\Omega$.

  Let us define the family of diffeomorphisms
  \[
    \Psi_i(x_1,\ldots,x_n):=(\mu_{i1}x_1(1+\varphi_{i1}),\ldots,\mu_{in}x_n(1+\varphi_{in})),\quad i=1,\ldots,p.
  \]
   Due to the property $\prod_{k=1}^n (1+\varphi_{ik})^{\omega_{jk}}=1$, it is commutative. As the infinite jets at $0$ of $\tilde\varphi_{ik}$ and $\varphi_{ik}$ are the same, the original family of diffeomorphisms $\Phi_1,\ldots,\Phi_p$ is still formally conjugate to the family of $\Psi_1,\ldots,\Psi_p$, and it follows by Chaperon's theorem that they are smoothly conjugate.
\end{proof}

Observe that hyperbolic systems are projectively hyperbolic and weakly hyperbolic, it follows naturally that
\begin{corollary}
  Let $(\Phi_1=\Phi,\Phi_2,\ldots,\Phi_p,F_1,\ldots,F_q)$ be a non-degenerate smooth integrable system of type $(p,q)$ on $\KK^n$ around $0$. If the system is hyperbolic, then the diffeomorphisms are smoothly conjugate to a smooth normal form of the form \eqref{good-nf} together with \eqref{intgnf} as in theorem \ref{thm:FormalNormalForm}.
\end{corollary}

\section{Real case}
In this section, we consider families of real commuting diffeomorphisms $\Phi_1,\ldots,\Phi_p$ on $(\RR^n,0)$. The coefficients of the Taylor expansion of $\Phi_i(x)$'s at the origin are real numbers. If the linear parts $\Phi^{(1)}_i$ is diagonalizable over $\RR$, then all preceding results hold true with the same proof. 

Here we are concerned with cases where  $\Phi^{(1)}_i=A_i$ are not diagonalizable over $\mathbb{R}$ but merely $\mathbb{C}$.   By the commutativity, one can decompose $\mathbb{R}^n= \oplus_{j=1}^l V_j\oplus \mathbb{R}^{n-2l}$
where each $V_j$ is a real plane left invariant by all $A_i$'s and such that  at least one of the $A_i|_{V_j}$'s is diagonalizable over $\mathbb{C}$ but not over $\mathbb{R}$. 
Under a  basis of vectors from eigenspaces, each $A_i$  becomes a  block diagonal matrix consisting of $l$ two by two blocks and $n-2l$ real numbers. Suppose the eigenvalues of $A_i|_{V_j}$ are $\mu_{ij}=u_{ij}+\sqrt{-1}v_{ij},\bar\mu_{ij}=u_{ij}-\sqrt{-1}v_{ij}$, then the $j$-th block of $A_i$ can be of the form $\begin{pmatrix}u_{ij}&-v_{ij}\\v_{ij}&u_{ij}\end{pmatrix}$ if the basis is well chosen.

Denote by $\mathcal{E}_j:=V_j\oplus\sqrt{-1}V_j$ the complexification of $V_j$, it is natural to get a canonical linear map  $A_i|_{\mathcal{E}_j}$. The complex vector $e=(\frac{1}{2},-\frac{1}{2}\sqrt{-1})$ in ${\mathcal{E}_j}$ is a common eigenvector belonging to $\mu_{ij}$ of $A_i|_{\mathcal{E}_j}$, i.e., $A_i|_{\mathcal{E}_j}e=\mu_{ij}e$ for $i=1,\ldots,p$.  Then $\bar e=(\frac{1}{2},\frac{1}{2}\sqrt{-1})$ is a common eigenvector of $\bar \mu_{ij}$ and $\mathcal{E}_j$ is isomorphic to the $\mathbb{C}$-vector space generated by $e, \bar e$.  Define 
$D_{ij}:= \begin{pmatrix}\mu_{ij}& 0\\ 0&\bar \mu_{ij}\end{pmatrix}$ and 
$P_j:=\begin{pmatrix}\frac{1}{2}&\frac{1}{2}\\ -\frac{1}{2}\sqrt{\scriptstyle{-}1}&\frac{1}{2}\sqrt{\scriptstyle{-}1}\end{pmatrix}$. 
Then for $i=1,\ldots,p$ and $j=1,\ldots,l$, we have
\[
A_{i}|_{\mathcal{E}_j} \, P_j = P_j D_{ij}.
\]
Let $P$ be the linear transformation on $\CC^n$ given by the block diagonal matrix consisting of $l$ copies of  $\begin{pmatrix}\frac{1}{2}&\frac{1}{2}\\ -\frac{1}{2}\sqrt{\scriptstyle{-}1}&\frac{1}{2}\sqrt{\scriptstyle{-}1}\end{pmatrix}$ and Identity of size $n-2l$ and $D_i$ be the linear transformation on $\CC^n$ given by the block diagonal matrix consisting of blocks $D_{i1},\ldots,D_{il}$ and Identity of size $n-2l$. Define $\rho$ to be the following involution $$
\rho(z_1,z_2,\ldots,z_{2l-1},z_{2l},z_{2l+1},\ldots,z_n):= (\bar z_2,\bar z_1,\ldots,\bar z_{2l},\bar z_{2l-1},\bar z_{2l+1},\ldots,\bar z_n),
$$ and denote by $c$  the complex conjugate $c(z_1,\ldots,z_n):=(\bar z_1,\ldots,\bar z_n)$. We easily have
$D_i\circ\rho=\rho\circ D_i \mbox{~for~} i=1,\ldots,p $ and
\begin{equation}
\label{eq:51}
P \rho=c P.
\end{equation}



Now let us consider the family $\{\tilde \Phi_i(z):=P^{-1}\Phi_i(P(z))\}_i$ of transformations of $\mathbb{C}^n$. Obviously,  $\tilde \Phi_i(z)$'s commute pairwise. If the family of $\Phi_i(z)$'s is non-degenerate, weakly non-resonant, (projectively, weakly) hyperbolic, then the family of $\tilde\Phi_i(z)$'s keeps these properties defined according to the eigenvalues which are the same of  $\tilde \Phi_i$'s and  $\Phi_i$'s. 

If $\Phi_i$'s have $q=n-p$   first integrals $F_1,\ldots,F_q$ functionally independent almost everywhere, then $\tilde F_j(z):=F_{j}(Pz)$'s are first integrals of the $\tilde \Phi_i$'s since 
$$
\tilde F_{j}(\tilde \Phi_i(z))= F_{j}(PP^{-1}\Phi_i(Pz))=F_{j}(Pz)=\tilde F_{j}(z).
$$
We also have $\tilde F_1,\ldots,\tilde F_q$ are functionally independent almost everywhere since $P$ is invertible. 

Hence, we get an integrable system $(\tilde \Phi_1,\ldots,\tilde \Phi_p,\tilde F_1,\ldots,\tilde F_q)$ on $\CC^n$ of type $(p,q)$.

Notice the coefficients of the Taylor series at the origin of $\Phi_i$'s are real, we have $c\circ\Phi_i\circ c=\Phi_i$ formally. With the help of the equations \eqref{eq:51} and its equivalent equation $P^{-1}c=\rho^{-1}P^{-1}=\rho P^{-1}$, we have formally
\[
\begin{aligned}
\tilde{\Phi}_i\circ\rho= P^{-1}\Phi_i(P\rho)
&= P^{-1}\circ c\circ c\circ\Phi_i(cP)\\
&= P^{-1}\circ (c\circ\Phi_i)\circ P
= \rho\circ P^{-1}\circ \Phi_i\circ P
= \rho\circ \tilde{\Phi}_i.
\end{aligned}
\]

This is the formal $\rho$-equivariant normal form theory (see lemma \ref{lem:PD-NF}): there exists a formal transformation $\Psi(z)$, tangent to identity at the origin, such that 
\begin{enumerate}
	\item $\Psi\circ \rho= \rho\circ \Psi$
	\item $\hat \Phi_i:=\Psi^{-1}\circ \tilde \Phi_i\circ \Psi$ is in the Poincar\'e-Dulac normal form, i.e., 
	$\hat \Phi_i\circ D_j=D_j\hat \Phi_i$.
\end{enumerate}

Now the proof of theorem \ref{thm:FormalNormalForm} works and we get the complexified integrable diffeomorphisms $\tilde \Phi_i$'s deduced from a real integrable system $(\Phi_1,\ldots, \Phi_p,F_1,\ldots,F_q)$ are formally conjugated by $\Psi$ to $\hat \Phi_i$'s  which are of the form \eqref{good-nf} together with \eqref{intgnf} as in theorem \ref{thm:FormalNormalForm} if the family is either projectively hyperbolic or infinitesimally integrable  with  a  weakly  non-resonant  family  of  generators.

\begin{lemma}
 The formal transformation $P\Psi P^{-1}$ is real in the sense that its coefficients are all real.
\end{lemma}
\begin{proof}
 The equation $cP\Psi P^{-1}c=P\rho\Psi\rho P^{-1}=P\Psi\rho^2 P^{-1}=P\Psi P^{-1}$ holds.
\end{proof}

Observe that $P\hat \Phi_i P^{-1} =P\Psi^{-1}\circ\tilde\Phi_i\circ\Psi P^{-1} =(P\Psi P^{-1})^{-1}\circ\Phi_i\circ(P\circ\Psi P^{-1})$, we have a  version of theorem \ref{thm:FormalNormalForm} for real diffeomorphisms having a linear part which is diagonal over $\mathbb{C}$ but not necessarily over $\mathbb{R}$. 
\begin{theorem}
Let $(\Phi_1=\Phi,\Phi_2,\ldots,\Phi_p,F_1,\ldots,F_q)$ be a formal non-degenerate discrete integrable system of type $(p,q)$ on $\RR^n$ at a common fixed point, say the origin $0$. If the family \{$\Phi_i^{(1)}\}$ is either projectively hyperbolic or infinitesimally integrable with a weakly non-resonant family of generators, then the family of real diffeomorphisms $\{\Phi_i\}$ is formally conjugated by the real formal transformation $P\Psi P^{-1}$ tangent to Identity to a real normal form $\{P\hat \Phi_i P^{-1}\}$ which is of the form
\beq\label{real-nf}
 (\frac{\hat\Phi_{i1}+\hat\Phi_{i2}}{2},\frac{\hat\Phi_{i1}-\hat\Phi_{i2}}{2\sqrt{-1}},\ldots,\frac{\hat\Phi_{i(2l-1)}+\hat\Phi_{i\,2l}}{2},\frac{\hat\Phi_{i(2l-1)}-\hat\Phi_{i\,2l}}{2\sqrt{-1}} ,\hat\Phi_{i(2l+1)},\ldots,\hat\Phi_{in})(z),
\eeq
where $\hat\Phi_{im}$ denotes the $m$-th component of $\hat\Phi_i$ which is the complex normal form of $\Phi_i$ as in theorem \ref{thm:FormalNormalForm} and $z=(z_1,\ldots,z_n)$ is defined as $z_{2j-1}=x_{2j-1}+x_{2j}\sqrt{-1},\,$  $z_{2j}=x_{2j-1}-x_{2j}\sqrt{-1}$ for $j=1,\ldots,l$ and $z_j=x_j$ for $j>2l$. 
\end{theorem}

\begin{proof}
From the expression above, we can see $\hat\Phi_{i(2j-1)}(z)$ and $\hat\Phi_{i(2j)}(z)$ have conjugate values.  Indeed, 
according to the properties of $\Psi$ and  $\tilde \Phi_i$ above, we have 
\beq\label{eq:53}
\rho\hat \Phi_i\rho= \rho\Psi\tilde \Phi_i\Psi^{-1}\rho=\Psi\rho\tilde \Phi_i\rho\Psi^{-1}=\Psi\tilde \Phi_i\Psi^{-1}=\hat\Phi_i.
\eeq
Therefore, by composition by $P$ on the left and by $P^{-1}$ on the right of the previous equation and by using \re{eq:51}, we obtain
$$
P\rho\hat \Phi_i\rho P^{-1}=cP\hat \Phi_i P^{-1}c=P\hat \Phi_i P^{-1},
$$
so that $P\hat \Phi_i P^{-1}$ is real.
Let $\sigma$ be the permutation mapping $2j-1$ to $2j$ and vice versa for $j\leqslant l$ and  fixing all integers from $2l+1$ to $n$. As $\mu_{im}$ and $\mu_{i\sigma(m)}$ for $m\leqslant2l$ are a pair of conjugate eigenvalues and $\mu_{im}$ are real for $m>2l$, any element $\gamma:=(\gamma_1,\ldots,\gamma_n)\in\mathcal R_m$ (cf. \re{eq:21}), we have $\gamma^\sigma:=(\gamma_{\sigma(1)},\ldots,\gamma_{\sigma(n)})\in\mathcal R_{\sigma(m)}$.
It follows that if  $\hat\Phi_{im,\gamma}w^\gamma$ is a resonant term in  $\hat\Phi_{im}(w)$, then $\overline{\hat\Phi_{im,\gamma^{\sigma}}}w^{\gamma^{\sigma}}$ is a term in $\hat\Phi_{i\sigma(m)}$ by \eqref{eq:53}. Hence, for $m\leqslant2l$, $\hat\Phi_{im}$ and $\hat\Phi_{i\sigma(m)}$ are a pair of conjugate functions of variables $(z_1,z_2=\bar z_1,\ldots,z_{2l-1},z_{2l}=\bar z_{2l-1},z_{2l+1},\ldots,z_n)$ and for $m>2l$, the values (not the functions) $\hat\Phi_{im}(z_1,\bar z_1,\ldots,z_{2l-1},\bar z_{2l-1},z_{2l+1},\ldots,z_n)$ are real since
\begin{eqnarray*}
\overline{\hat\Phi_{im}(\bar z_2,\bar z_1,\ldots,\bar z_{2l},\bar z_{2l-1},\bar z_{2l+1},\ldots,\bar z_n)}&=& \overline{\hat\Phi_{im}( z_1,z_2,\ldots, z_{2l-1},z_{2l},z_{2l+1},\ldots, z_n)}\\
&=& \hat\Phi_{im}( z_1,z_2,\ldots, z_{2l-1},z_{2l},z_{2l+1},\ldots, z_n).
\end{eqnarray*}
\end{proof}

%
%

Let $(\Phi_1=\Phi,\Phi_2,\ldots,\Phi_p,F_1,\ldots,F_q)$ be a formal non-degenerate discrete integrable system of type $(p,q)$ on $\RR^n$ at a common fixed point, say the origin $0$. We assume that the family of its linear parts $\{A_jx\}$ is either projectively hyperbolic or infinitesimally integrable with a weakly non-resonant family of generators. 
Assume furthermore that the commuting family of real diffeomorphisms $\{\Phi_i\}$ satisfies $A_j\Phi_i=\Phi_i A_j$, for all $i,j$. Here we assume that the matrices $A_j=PD_jP^{-1}$ are simultaneously diagonalizable over $\mathbb{C}$ but not necessarily over $\mathbb{R}$. Then, we have  $(P^{-1}A_j P)(P^{-1}\Phi_i P)=(P^{-1}\Phi_i P)(P^{-1}A_j P)$. Hence, the family $\{P^{-1}\Phi_i P\}$ is in Poincar\'e-Dulac normal form as it commutes with the family of its linear part $\{D_j\}$
Since the family $(P^{-1}\Phi_1P,\ldots,P^{-1}\Phi_pP,F_1\circ P,\ldots,F_q\circ P)$ satisfies assumption of \rt{thm:FormalNormalForm}, then $P^{-1}\Phi_iP$ is of the form \re{good-nf} with \re{intgnf}, for all $i$. Therefore,  $\Phi_i$ is of the form \re{real-nf} in which $\hat \Phi_i$ have to be replaced by $P^{-1}\Phi_iP$.

\end{document}